\documentclass{article}

\usepackage{geometry}
\usepackage{graphicx} 
\usepackage{epstopdf} 

\usepackage[colorlinks, hypertexnames=false]{hyperref}
\hypersetup{linkcolor=blue, citecolor=red} 

\usepackage{bm}
\usepackage{amsmath} 
\usepackage{amssymb} 
\usepackage{stmaryrd} 
\usepackage{multirow}

\usepackage{float}
\usepackage{amsthm} 
\usepackage{mathtools} 

\usepackage[dvipsnames]{xcolor}
\definecolor{red}{rgb}{1.00,0.00,0.00}
{\numberwithin{equation}{section}
\setlength{\parindent}{1em}

\newtheorem{theorem}{Theorem}[section]
\newtheorem{lemma}{Lemma}[section]
\newtheorem{remark}{Remark}[section]

\renewcommand{\footnotesize}{\scriptsize}
\newcommand{\normmm}[1]{{\left\vert\kern-0.25ex\left\vert
\kern-0.25ex\left\vert #1
    \right\vert\kern-0.25ex\right\vert\kern-0.25ex\right\vert}}
\newcommand{\avg}[1]{\ensuremath{\left\{\!\!\left\{ #1\right\}\!\!\right\}}}
\newcommand{\jump}[1]{\ensuremath{\left\llbracket #1\right\rrbracket}}

\newcommand{\tdiv}{\mathrm{div}\,}

\geometry{left=3cm,right=3cm,top=2.5cm,bottom=2cm}
\newcommand{\bs}{\boldsymbol}

\begin{document}           

\title{A strongly mass conservative method for the coupled Brinkman-Darcy flow and transport}
\author{Lina Zhao\footnotemark[1]\quad and\quad
Shuyu Sun\footnotemark[2]}
\renewcommand{\thefootnote}{\fnsymbol{footnote}}
\footnotetext[1]{Department of Mathematics, City University of Hong Kong, Kowloon Tong, Hong Kong SAR, China. ({linazha@cityu.edu.hk}).}
\footnotetext[2]{Computational Transport Phenomena Laboratory (CTPL), Physical Science and Engineering Division (PSE), King Abdullah University of Science andTechnology, Thuwal, 23955-6900, Saudi Arabia. Corresponding author. (shuyu.sun@kaust.edu.sa).}
\maketitle

\begin{abstract}
In this paper, a strongly mass conservative and stabilizer free scheme is designed and analyzed for the coupled Brinkman-Darcy flow and transport. The flow equations are discretized by using a strongly mass conservative scheme in mixed formulation with a suitable incorporation of the interface conditions. In particular, the interface conditions can be incorporated into the discrete formulation naturally without introducing additional variables. Moreover, the proposed scheme behaves uniformly robust for various values of viscosity. A novel upwinding staggered DG scheme in mixed form is exploited to solve the transport equation, where the boundary correction terms are added to improve the stability. A rigorous convergence analysis is carried out for the approximation of the flow equations. The velocity error is shown to be independent of the pressure and thus confirms the pressure-robustness. Stability and a priori error estimates are also obtained for the approximation of the transport equation; moreover, we are able to achieve a sharp stability and convergence error estimates thanks to the strong mass conservation preserved by our scheme. In particular, the stability estimate depends only on the true velocity on the inflow boundary rather than on the approximated velocity. Several numerical experiments are presented to verify the theoretical findings and demonstrate the performances of the method.

\end{abstract}


\pagestyle{myheadings} \thispagestyle{plain}
\markboth{ZhaoChung} {The coupling of Brinkman-Darcy flows and transport}

\section{Introduction}

%
%

Coupling Brinkman and Darcy models describes the interaction of flow and transport phenomena in two different domains separated by an interface. This model has been used in the hydrology and biological applications and typical examples include subsurface flow, hydraulic fractures and perfusion of soft living tissues. A great amount of effort has been devoted to the devising of efficient numerical schemes for the coupled flow and transport. In \cite{SunWheeler05}, primal discontinuous Galerkin methods with interior penalty are developed to solve the coupled system of flow and reactive transport. In \cite{Vassilev09}, a mixed finite element element is exploited to approximate the Stokes-Darcy system and a local discontinuous Galerkin method is used to discretize the transport equation. In \cite{Rui17}, a stabilized mixed finite element method in conjunction with velocity-pressure-concentration formulation is exploited to discretize the coupled Stokes-Darcy flow and transport. In \cite{Alvarez16,Alvarez20}, a primal mixed finite element method in conjunction with vorticity-velocity-pressure formulation is used for the discretization of the Brinkman-Darcy flow and a conforming finite element method is used for the discretization of the nonlinear transport equation. In the works presented in \cite{Zhang18,Ervin19}, the authors are devoted to the analysis of partitioned time stepping methods for the conforming discretizations on the two subdomains. In addition to the aforementioned works, there are also a surge of works that have been dedicated to the devising and analysis of numerical schemes for the Stokes-Darcy flow and/or transport, see, e.g., \cite{Layton02,Dawson04,Riviere05,Chidyagwai09,Ervin09,Gatica11,Lipnikov14,Lee16,Fu18,Zhao20}.

There are some typical difficulties in the devising of an efficient and accurate scheme for the coupled Brinkman-Darcy flow and transport problem considered in this paper. In a nutshell, the Brinkman equations model both the Stokes problem and the Darcy problem in porous media, and the devising of a uniformly stable scheme for both the Stokes problem and the Darcy problem is challenging due to the different inherent natures of these two equations. It becomes even more challenging for the coupled Brinkman-Darcy flow as one also needs to balance the Brinkman problem and Darcy problem with a suitable treatment for the interface conditions. The way to enforce the interface conditions also needs a careful design.
Another typical issue posed is the mass conservation, which is of great importance in the context of transport equation. The devising of a numerical scheme which can overcome the aforementioned difficulties is a challenging task and it hinges on a dedicate balancing of the finite element spaces used. Classical techniques that have been developed to overcome the difficulties encountered in the design of uniformly robust schemes for the Brinkman problem include nonconforming methods with $H(\tdiv;\Omega)$-conforming velocity \cite{Mardal02,Guzman12}, $H(\tdiv;\Omega)$-conforming discontinuous Galerkin methods \cite{Konno11,Konno12} and parameter free $H(\tdiv;\Omega)$-conforming HDG methods \cite{FuQiu18}. A common ingredient shared by these methods is to relax the tangential continuity of velocity; indeed, a $H(\tdiv;\Omega)$-conforming space is employed to approximate the velocity. An alternative approach is to modify the right hand side of classical finite element methods by using divergence-free velocity reconstruction operator, and a lot of works has been developed in this direction, see, for example, \cite{Linke12,Linke14,Lederer17,Frerichs20,Mu20,Apel21}. Devising a uniformly robust scheme for the coupled Brinkman-Darcy flow and meanwhile preserving the global mass conservation of the method without resorting to additional variables are not easy tasks, and one must carefully design the approximation spaces so that the interface conditions can be incorporated into the discrete formulation naturally.

Therefore, the purpose of this paper is to devise and analyze a strongly mass conservative scheme of arbitrary polynomial orders for the coupled Brinkman-Darcy flow and transport. The Brinkman equations are discretized by using a carefully designed staggered DG method in conjunction with velocity gradient-velocity-pressure formulation and the Darcy equations are discretized by using mixed finite element method, this choice of spaces makes our scheme capable of handling the interface conditions without resorting to additional variables. More precisely, the interface conditions can be imposed into the discrete formulation naturally by replacing the Brinkman's normal velocity by the Darcy's normal velocity.
It should be noted that the choice of the spaces for the Brinkman equation should be carefully designed so that the resulting scheme is uniformly robust for various feasible values of viscosity. The key ingredient is to use a locally $H(\tdiv;\Omega)$-conforming space to approximate the Brinkman velocity. The transport equation is discretized by using an upwinding staggered DG method, where the boundary correction terms are exploited to improve the stability. The proposed scheme possesses many desirable features, which makes it attractive. First, it is globally mass conservative and the interface conditions can be imposed without resorting to additional variables. Second, it is uniformly robust for various feasible values of viscosity. Third, the normal continuity of velocity is satisfied exactly at the discrete level. Fourth, no penalty term or stabilization term is needed, which is advantageous over other DG methods since choosing a suitable stabilization parameter could be tricky for certain situation. A rigorous convergence analysis is carried out for the Brinkman-Darcy flow equations. In addition, we also analyze the stability and convergence error estimates for the concentration and the diffusive flux in the transport equation, where the stability estimate depends only on the true velocity on the inflow boundary rather than on the approximated velocity. The error results from a combination of the upwinding staggered DG discretization error and the error from the discretization of the Brinkman-Darcy velocity. The resulting convergence error estimates for the transport equation are of order $\mathcal{O}(h^{k+1})$, where $k$ is the polynomial order used for the discretization. To the best of our knowledge, our proposed method appears to be the first in the literature, that offers a robust behavior with respect to viscosity for the coupling of Brinkman-Darcy flow and transport without resorting to additional variables to enforce the interface conditions.

The rest of the paper is organized as follows. In the next section, we present the model flow-transport problem. The flow discretization is given in Section~\ref{sec:flow}, and the corresponding convergence error estimates are provided in Section~\ref{sec:error}. The transport discretization and its error analysis are presented in Section~\ref{sec:transport}. Several numerical experiments are presented in Section~\ref{sec:numerical}. Finally, the conclusions are given in Section~\ref{sec:conclusion}.

\section{Model problem}

In our model we consider a fluid region $\Omega_B\subset \mathbb{R}^2$ in which the flow is governed by the Brinkman equations \eqref{eq:Brinkman1}-\eqref{eq:Brinkman3} and a porous medium $\Omega_D\subset \mathbb{R}^2$ in which Darcy flow equations \eqref{eq:Darcy1}-\eqref{eq:Darcy2} hold. These two regions are separated by an interface $\Gamma$, through which exchange of fluid velocities and pressures occurs; see Figure~\ref{fig:domain} for an illustration of the computational domain.
The flow model reads as follows:
\begin{alignat}{2}
\epsilon^{-1}\bm{L}&=\nabla \bm{u}_B&&\quad \mbox{in}\;\Omega_B\label{eq:Brinkman1},\\
- \nabla \cdot \bm{L}+\alpha \bm{u}_B+\nabla p_B& = \bm{f}_B&&\quad\mbox{in}\;\Omega_B\label{eq:Brinkman2},\\
\nabla \cdot \bm{u}_B&=0&&\quad \mbox{in}\;\Omega_B\label{eq:Brinkman3}
\end{alignat}
and
\begin{alignat}{2}
\bm{u}_D+K_D\nabla p_D&= \bm{f}_D&&\quad \mbox{in}\;\Omega_D\label{eq:Darcy1},\\
\nabla\cdot \bm{u}_D&=f&&\quad \mbox{in}\;\Omega_D\label{eq:Darcy2}.
\end{alignat}
Here $\epsilon$ is the effective viscosity constant, the inverse of $\alpha$ is the permeability tensor constant divided by viscosity and $K_D$ is a symmetric and positive definite permeability tensor divided by viscosity (which is also called hydraulic conductivity). $\bm{f}_B\in L^2(\Omega_B)^2$, $\bm{f}_D\in L^2(\Omega_D)^2$ and $f\in L^2(\Omega_D)$ are given data.

We assume that $\Omega$ has a Lipschitz continuous boundary split into two disjoint sub-boundaries with positive measure, i.e., $\partial \Omega=\Gamma_B\cup \Gamma_D$, where $\Gamma_B=\partial \Omega_B\backslash \Gamma$ and $\Gamma_D=\partial \Omega_D\backslash \Gamma$.
Following \cite{Alvarez20}, we adopt the following interface conditions:
\begin{alignat}{2}
\bm{u}_D\cdot\bm{n}_B&=\bm{u}_B\cdot\bm{n}_B&&\quad\mbox{on}\;\Gamma,\label{eq:interface1}\\
p_D&=p_B &&\quad\mbox{on}\;\Gamma\label{eq:interface2},
\end{alignat}
where $\bm{n}_B$ denotes the unit outward normal vector to $\Omega_B$. Similarly, we use $\bm{n}_D$ to represent the unit outward normal vector to $\Omega_D$.
To close the system, we define the following boundary conditions
\begin{align*}
\bm{L}\bm{n}_B =\bm{0}\quad \mbox{on}\;\Gamma_B\cup \Gamma,\quad \bm{u}_B\cdot\bm{n}_B=g_1\quad \mbox{on}\;\Gamma_B, \quad \bm{u}_D\cdot\bm{n}_D=g_2\quad \mbox{on}\;\Gamma_D.
\end{align*}
In addition, we require $\int_{\Omega_B}p_B\;dx=0$ to ensure the unique solvability.

%
The Brinkman-Darcy flow system is coupled with the transport equation in $\Omega=\Omega_B\cup \Omega_D$
\begin{align}
\phi c_t+\nabla\cdot (c\bm{u}-K\nabla c)=\phi s+\hat{c} f^+ - c f^-\quad \forall (x,t)\in \Omega\times(0,T),\label{eq:modeltransport}
\end{align}
where $T$ is the final simulation time, $c(x,t)$ is the concentration of a certain chemical component of interest, $0<\phi_*\leq\phi(x)\leq\phi^*$ is the porosity of the medium in $\Omega_D$ (it is set to 1 in $\Omega_B$), $K(x,t)$ is the diffusion-dispersion tensor assumed to be symmetric and positive definite with smallest and largest eigenvalues $K_{\text{min}}$ and $K_{\text{max}}$, respectively, $s(x,t)$ is a source term, and $\bm{u}$ is the velocity field defined by $\bm{u}\vert_{\Omega_i}=\bm{u}_i,i=B,D$. $\hat{c}$ in the source term is the injected concentration. In addition, we let $f^+=\max\{f,0\}$ and $f^-=\max\{-f,0\}$, it follows $f=f^+-f^-$. We remark that $f$ is only defined for $\Omega_D$ and we can simply take $f=0$ in $\Omega_B$. In general, the diffusion-dispersion tensor can be a function of the Darcy velocity; for simplicity of discussion in this paper, we assume that the diffusion-dispersion tensor is a given value. The model is completed by the initial condition
\begin{align*}
c(x,0)=c^0(x)\quad\forall x\in \Omega
\end{align*}
and the boundary conditions
\begin{align*}
(c\bm{u}-K\nabla c)\cdot\bm{n}&=(c_{\text{in}}\bm{u})\cdot\bm{n}\quad\mbox{on}\;\Gamma_{\textnormal{in}},\\
(K\nabla c)\cdot\bm{n}&=0\quad\mbox{on}\;\Gamma_{\text{out}}.
\end{align*}
Here, $c_{\text{in}}$ is the inflow concentration, $\Gamma_{\text{in}}:=\{x\in \partial \Omega:\bm{u}\cdot\bm{n}<0\}$ and $\Gamma_{\text{out}}:=\{x\in \partial \Omega:\bm{u}\cdot\bm{n}\geq 0\}$, and $\bm{n}$ is the unit outward normal vector to $\partial \Omega$.

\begin{figure}[t]
\centering
\includegraphics[width=0.45\textwidth]{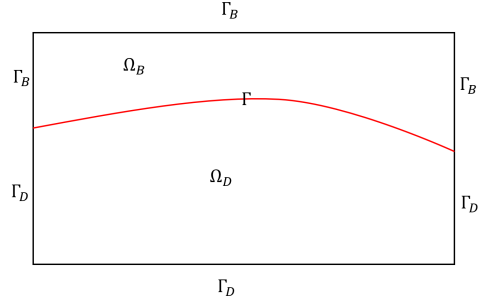}
\caption{The profile of the computational domain.}
\label{fig:domain}
\end{figure}

The following compatibility condition holds
\begin{align*}
\int_{\Gamma_B}g_1\;ds+\int_{\Gamma_D}g_2\;ds=\int_{\Omega_D} f\;dx.
\end{align*}
We can infer from Stokes' theorem that
\begin{align*}
\int_{\Gamma_D}g_2\;ds+\int_{\Gamma} \bm{u}_D\cdot\bm{n}_D\;ds=\int_{\Omega_D} f\;dx,
\end{align*}
thereby, we have
\begin{align}
\int_{\Gamma} \bm{u}_D\cdot\bm{n}_D\;ds=\int_{\Gamma_B}g_1\;ds.\label{eq:uDg1}
\end{align}
Before closing this section, we introduce some notation that will be used throughout the paper. Let $D\subset \mathbb{R}^d,d=1,2$. By $(\cdot,\cdot)_D$, we denote the standard scalar product in $L^2(D):(p,q)_D:=\int_D p \;q\;dx$. When $D$ coincides with $\Omega$, the subscript $\Omega$ will be dropped. We use the same notation for the scalar product in $L^2(D)^2$ and in $L^2(D)^{2\times 2}$. More precisely, $(\bm{\xi},\bm{w})_D:=\sum_{i=1}^2 (\xi^i,w^i)$ for $\bm{\xi},\bm{w}\in L^2(D)^2$ and $(\underline{\psi},\underline{\zeta})_D:=\sum_{i=1}^2\sum_{j=1}^2 (\psi^{i,j},\zeta^{i,j})_D$ for $\underline{\psi},\underline{\zeta}\in L^2(D)^{2\times 2}$. The associated norm is denoted by $\|\cdot\|_{0,D}$.
Given an integer $m\geq 0$ and $n\geq 1$, $W^{m,n}(D)$ and $W_0^{m,n}(D)$ denote the usual Sobolev space provided the norm
and semi-norm $\|v\|_{W^{m,n}(D)}=\{\sum_{|\ell|\leq m}\|D^\ell
v\|^n_{L^n(D)}\}^{1/n}$, $|v|_{W^{m,n}(D)}=\{\sum_{|\ell|=
m}\|D^\ell v\|^n_{L^n(D)}\}^{1/n}$. If $n=2$ we usually write
$H^m(D)=W^{m,2}(D)$ and $H_0^m(D)=W_0^{m,2}(D)$,
$\|v\|_{H^m(D)}=\|v\|_{W^{m,2}(D)}$ and $|v|_{H^m(D)}=|v|_{W^{m,2}(D)}$.
In the sequel, we use $C$ to represent a generic positive constant independent of the mesh size which may have different values at different occurrences.

For $\mathcal{O}\subset \mathbb{R}^2$, we define
\begin{align*}
H(\text{div};\mathcal{O}):=\{\bm{v}\in L^2(\mathcal{O})^2: \nabla \cdot \bm{v}\in L^2(\mathcal{O})\}
\end{align*}
whose norm is given by
\begin{align*}
\|\bm{v}\|_{\text{div},\mathcal{O}}:=\Big(\|\bm{v}\|_{0,\mathcal{O}}^2+\|\nabla \cdot \bm{v}\|_{0,\mathcal{O}}^2\Big)^{\frac{1}{2}}.
\end{align*}
In addition, we define the subspace $H_{0,\Gamma}(\text{div};\mathcal{O})$ by
\begin{align*}
H_{0,\Gamma}(\text{div};\mathcal{O}):=\{\bm{v}\in H(\text{div};\mathcal{O}); \bm{v}\cdot \bm{n}=0\;\textnormal{on}\;\Gamma\}.
\end{align*}
For $0\leq s<\infty$, we let
\begin{align*}
H^s(\text{div};\mathcal{O}):=\{\bm{v}\in L^2(\mathcal{O})^2 \;|\; \nabla \cdot \bm{v}\in H^s(\mathcal{O})\}.
\end{align*}
The following space is also defined for later use
\begin{align*}
L^2_0(\mathcal{O}):=\{q\in L^2(\mathcal{O}); \int_\mathcal{O} q\;dx=0\}
\end{align*}
and for $\omega\subset \mathbb{R}$, we define
\begin{align*}
H^1_{0,\omega}(\Omega_i):=\{q\in H^1(\Omega_i); q=0\;\mbox{on}\;\omega\}.
\end{align*}

\section{The new scheme for Brinkman-Darcy flow}\label{sec:flow}

In this section, we will derive the discrete formulation for the coupled Brinkman-Darcy flow. The proposed method should be uniformly robust with respect to viscosity and stabilizer free. The key idea lies in a delicate balancing of the finite element spaces involved.


First, we introduce the meshes and the spaces exploited in the definition of the new scheme. To simplify the presentation,
 we employ the same types of meshes for $\Omega_B$ and $\Omega_D$. Following \cite{Zhao2018,ZHAO2019}, we first let $\mathcal{T}_{u,i}$ ($i=B,D$) be the
initial partition of the domain $\Omega_i$ into non-overlapping triangular or quadrilateral meshes. We require that $\mathcal{T}_{u,i}$ be aligned with $\Gamma$.
We let $\mathcal{F}_{pr,i}$ be the set of all edges excluding the interface edges in the initial partion $\mathcal{T}_{u,i}$ and $\mathcal{F}_{pr,i}^{0}\subset \mathcal{F}_{pr,i}$ be the
subset of all interior edges of $\Omega_i$. In addition, we use $\mathcal{F}_{h,\Gamma}$ to represent the set of edges lying on the interface $\Gamma$.
For each primal element $E$ in the initial partition $\mathcal{T}_{u,i}$, we select an interior point $\nu$ and
create new edges by
connecting $\nu$ to all the vertices of the primal element. For simplicity, we select $\nu$ as the center point.
This process will divide $E$ into the union of subtriangles, where the subtriangle is denoted as $\tau$, and we rename
the union of these triangles by $S(\nu)$. We remark that $S(\nu)$ is the triangular or rectangular mesh in the initial partition.
Moreover, we will use $\mathcal{F}_{dl,i}$ to denote the set of all the new edges generated by this subdivision process and
use $\mathcal{T}_{h,i}$ to denote the resulting quasi-uniform triangulation,
on which our basis functions are defined. Here the triangulation $\mathcal{T}_{h,i}$ satisfies standard mesh regularity assumption (cf. \cite{Ciarlet78}) and we define $\mathcal{T}_{h}=\mathcal{T}_{h,B}\cup \mathcal{T}_{h,D}$.
In addition, we let $\mathcal{F}_i:=\mathcal{F}_{pr,i}\cup \mathcal{F}_{dl,i}$, $\mathcal{F}_i^{0}:=\mathcal{F}_{pr,i}^{0}\cup \mathcal{F}_{dl,i}$, $\mathcal{F}_{pr}:=\mathcal{F}_{pr,B}\cup \mathcal{F}_{pr,D}\cup \mathcal{F}_{h,\Gamma}$, $\mathcal{F}_{pr}^0:=\mathcal{F}_{pr,B}^0\cup \mathcal{F}_{pr,D}^0\cup \mathcal{F}_{h,\Gamma}$ and $\mathcal{F}_{dl}:=\mathcal{F}_{dl,B}\cup \mathcal{F}_{dl,D}$. For each triangle
$\tau\in \mathcal{T}_{h,i}$, we let $h_\tau$ be the diameter of
$\tau$ and $h_i=\max\{h_\tau, \tau\in \mathcal{T}_{h,i}\}$, and we define $h=\max\{h_B,h_D\}$. Also, we let $h_e$ denote the length of edge $e\in \mathcal{F}_i$.
This construction is illustrated in Figure~\ref{grid},
where the black solid lines are edges in $\mathcal{F}_{pr,i}$
and the red dotted lines are edges in $\mathcal{F}_{dl,i}$. For each interior edge $e\in \mathcal{F}_{pr,i}^0$, we use $D(e)$ to denote the union of the two triangles in $\mathcal{T}_{h,i}$ sharing the edge $e$,
and for each boundary edge $e\in(\mathcal{F}_{pr,i}\cup \mathcal{F}_{h,\Gamma})\backslash\mathcal{F}_{pr,i}^0$, we use $D(e)$ to denote the triangle in $\mathcal{T}_{h,i}$ having the edge $e$,
see Figure~\ref{grid}.

For each edge $e$, we define
a unit normal vector $\bm{n}_{e}$ as follows: If $e\in \mathcal{F}_i\setminus \mathcal{F}_i^{0}$, then
$\bm{n}_{e}$ is the unit normal vector of $e$ pointing towards the outside of $\Omega_i$. If $e\in \mathcal{F}_i^{0}$, an
interior edge, we then fix $\bm{n}_{e}$ as one of the two possible unit normal vectors on $e$.
When there is no ambiguity,
we use $\bm{n}$ instead of $\bm{n}_{e}$ to simplify the notation.
For $k\geq 1$, $\tau\in \mathcal{T}_h$ and $e\in \mathcal{F}_h$, we define $P^k(\tau)$ and $P^k(e)$ as the spaces of polynomials of degree up to order $k$ on $\tau$ and $e$, respectively. For a scalar or vector function $v$ belonging to the broken Sobolev space, its jump and average on $e\in
\mathcal{F}_{i}$ are defined as
\begin{equation*}
\jump{v}_e:=v_{1}-v_{2},\quad \avg{v}_e:=\frac{v_{1}+v_{2}}{2},
\end{equation*}
where $v_j=v_{\tau_j},j=1,2$ and $\tau_{1}$, $\tau_{2}$ are the
two triangles in $\mathcal{T}_{h,i}$ having the edge $e$. For the boundary edges, i.e., edges belong to $\Gamma\cup \partial \Omega$, we simply define $\jump{v}_e=v_{1}$ and $\avg{v}_e=v_{1}$. We can omit the subscript $e$ when it is clear which edge we are referring to. In the following, we use $\nabla_h$ and $\tdiv_h$ to represent the element-wise defined gradient and divergence operators.

\begin{figure}[t]
\centering
\includegraphics[width=0.45\textwidth]{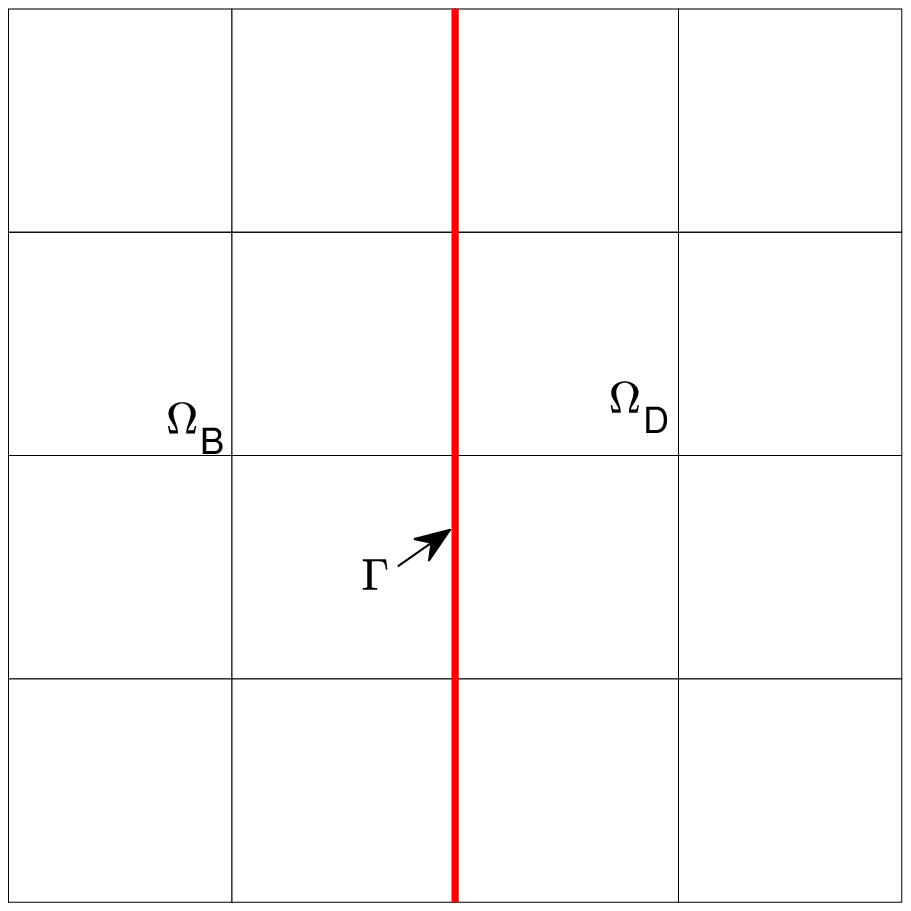}
\includegraphics[width=0.45\textwidth]{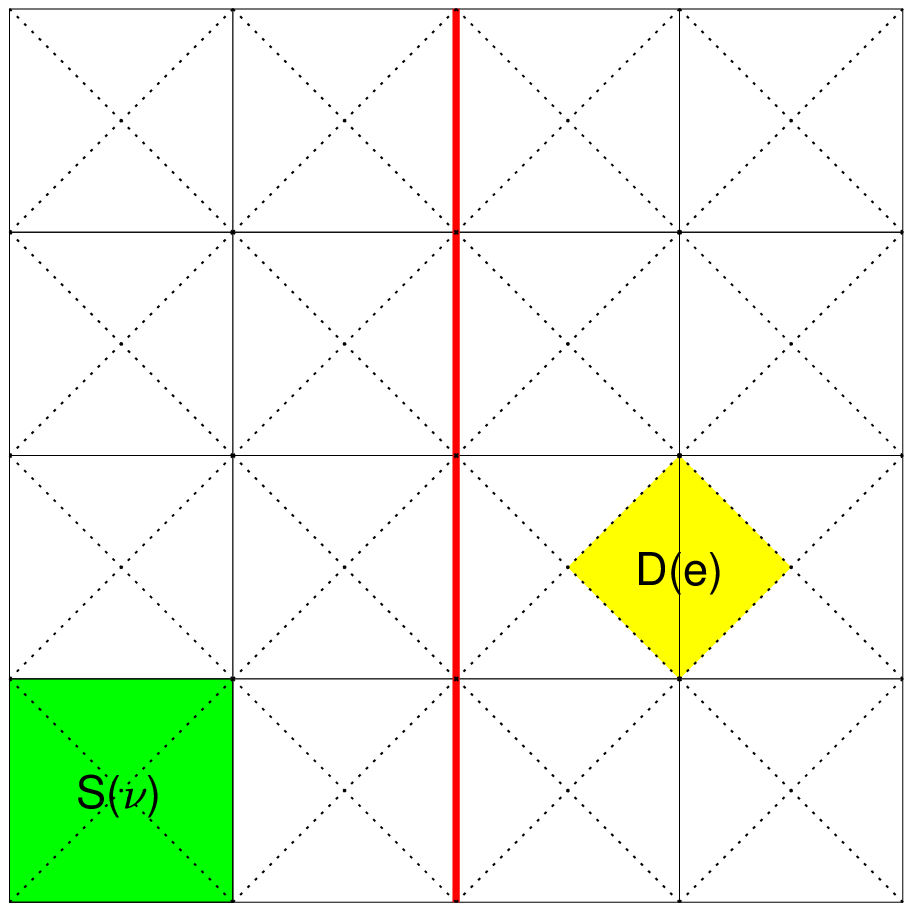}
\caption{Schematics of the meshes. Primal meshes (left), dual meshes and simplicial meshes (right). The solid lines represent the primal edges and the dashed lines represent the dual edges.}
\label{grid}
\end{figure}

Now we are ready to define the finite element spaces that will be used for the numerical approximation. First, the locally $H(\text{div};\Omega_B)$-conforming space for the approximation of $\bm{u}_B$ is defined by
\begin{equation*}
H_h^B:=\{\bm{v} \: : \:  \bs{v} \vert_{\tau} \in P^k(\tau)^2, \forall \tau \in \mathcal{T}_{h,B}; \jump{\bs{v} \cdot \bm{n}}_e=0 ,\forall  e \in \mathcal{F}_{dl,B}\}.
\end{equation*}
The finite dimensional space used for the approximation of $\bm{L}$ is defined by
\begin{align*}
W_h^B:&=\{\bm{G} \: : \:  \bm{G} \vert_{\tau} \in P^k(\tau)^{2\times 2}, \forall \tau \in \mathcal{T}_{h,B}; \jump{\bm{G} \bs{n}}_e=\bm{0},\forall  e\in \mathcal{F}_{pr,B}^0,\\
&\; \jump{(\bm{G}\bm{n})\cdot \bm{t}}_e=0, \forall e\in \mathcal{F}_{dl,B}; \bm{G}\bm{n}=\bm{0}\;\mbox{on}\;\partial \Omega_B\}.
\end{align*}
The locally $H^1(\Omega_B)$-conforming space for the approximation of $p_B$ is defined by
\begin{equation*}
Q_h^B:=\{q \: : \:  q \vert_{\tau} \in P^k(\tau), \forall \tau \in \mathcal{T}_{h,B}; \jump{q}_e=0,\forall e\in \mathcal{F}_{pr,B}^0\}
\end{equation*}
and
\begin{align*}
Q_h^0:=\{q\: : \: q\in Q_h^B;  \int_{\Omega_B} q\;dx=0\}.
\end{align*}

For later analysis, we define the following mesh dependent semi-norm for any $q_B\in Q_h^B$ and $\bm{v}_B\in H_h^B$
\begin{align*}
\|q_B\|_h^2:&=\sum_{e\in \mathcal{F}_{dl,B}}h_e^{-1}\|\jump{q_B}\|_{0,e}^2+\sum_{\tau\in \mathcal{T}_{h,B}}\|\nabla q_B\|_{0,\tau}^2,\\
\|\bm{v}_B\|_{Z}^2:&=\sum_{\tau\in \mathcal{T}_{h,B}}\|\nabla \bm{v}_B\|_{0,\tau}^2+\sum_{e\in \mathcal{F}_{pr,B}^0\cup\mathcal{F}_{h,\Gamma}}h_e^{-1}\|\jump{\bm{v}_B}\|_{0,e}^2+\sum_{e\in \mathcal{F}_{dl,B}}h_e^{-1}\|\jump{(\bm{v}_B\cdot \bm{t})\bm{t}}\|_{0,e}^2.
\end{align*}
Note that $\|q_B\|_h$ is a norm on $Q_h^0$.


Following \cite{ChungEngquist09}, we specify the degrees of freedom for $Q_h^B$ as follows:
    \begin{itemize}
        \item[(SD1)]
            For $e\in \mathcal{F}_{pr,B}\cup \mathcal{F}_{h,\Gamma}$, we have
            \begin{equation}
                \phi_e(q) :=( q,p_k)_e\quad \forall p_k\in P^k(e).\label{eq:dof1}
            \end{equation}
        \item[(SD2)]
            For each $\tau\in \mathcal{T}_{h,B}$, we define
            \begin{equation}
                \phi_\tau(q) := (q,p_{k-1})_\tau\quad \forall p_{k-1}\in P^{k-1}(\tau).\label{eq:dof2}
            \end{equation}
    \end{itemize}

We employ mixed finite element method for the Darcy region. To this end, we define the finite element subspace $H_h^D$ for the approximation of $\bm{u}_D$ by
\begin{align*}
H_h^D:=\{\bm{v}\in H(\text{div};\Omega_D): \bm{v}\vert_\tau\in P^k(\tau)^2,\forall \tau\in \mathcal{T}_{h,D}\},
\end{align*}
where we can take the Brezzi-Douglas-Marini (BDM) space (cf. \cite{Brezzi85}).

The finite element subspace for $p_D$ is given by the piecewise polynomials of degree $k-1$, that is,
\begin{align*}
Q_h^D:=\{q\in L^2(\Omega_D): q \vert_\tau \in P^{k-1}(\tau), \forall \tau\in \mathcal{T}_{h,D}\}.
\end{align*}

In addition, let $\Pi_{h,\Gamma_i},i=B,D$ be the piecewise $L^2$-projection onto $P^k(e)$ for all $e$ belonging to $\Gamma_i$ such that for all $\psi\in L^2(\Gamma_i)$
\begin{align*}
(\psi-\Pi_{h,\Gamma_i}\psi,\chi)_e=0\quad \forall \chi\in P^k(e),e\in \Gamma_i
\end{align*}
and define
\begin{align*}
H_{h,g_2}^D:=\{\bm{v}\in H_h^D,\bm{v}\cdot\bm{n}=\Pi_{h,\Gamma_D}g_2\;\mbox{on}\;\Gamma_D\}.
\end{align*}
Now we are ready to derive the discrete formulation. Multiplying \eqref{eq:Brinkman1} by a test function  $\bm{G}_B\in W_h^B$ and performing integration by parts yield
\begin{align*}
(\epsilon^{-1}\bm{L},\bm{G}_B)_{\Omega_B}&=(\nabla \bm{u}_B,\bm{G}_B)_{\Omega_B}\\
&=\sum_{\tau\in \mathcal{T}_{h,B}}\Big((\bm{G}_B \bm{n}, (\bm{u}_B\cdot \bm{n})\bm{n}+(\bm{u}_B\cdot \bm{t})\bm{t})_{\partial \tau}-(\text{div}\, \bm{G}_B, \bm{u}_B)_\tau\Big)\\
&=\sum_{e\in \mathcal{F}_{dl,B}}(\jump{\bm{G}_B \bm{n}}, (\bm{u}_B\cdot \bm{n})\bm{n})_e-(\tdiv_h \bm{G}_B, \bm{u}_B)_{\Omega_B},
\end{align*}
where we use the decomposition $\bm{u}=(\bm{u}\cdot\bm{n})\bm{n}+(\bm{u}\cdot\bm{t})\bm{t}$ and the facts that $[(\bm{G}_B\bm{n})\cdot\bm{t}]\mid_e=0,\forall e\in \mathcal{F}_{dl,B}$, $[\bm{G}_B\bm{n}]\mid_e=\bm{0},\forall e\in\mathcal{F}_{pr,B}^0 $ and $\bm{G}_B\bm{n}_B=\bm{0}$ on $\partial \Omega_B$.

Multiplying \eqref{eq:Brinkman2} by a test function $\bm{v}_B\in H_h^B$ and performing integration by parts lead to
\begin{align*}
&- (\text{div}\, \bm{L},\bm{v}_B)_{\Omega_B}+ (\alpha\bm{u}_B,\bm{v}_B)_{\Omega_B}+(\nabla p_B,\bm{v}_B)_{\Omega_B}\\
& =-\sum_{e\in \mathcal{F}_{pr,B}^0}(\bm{L} \bm{n}, \jump{\bm{v}_B})_e-\sum_{e\in \mathcal{F}_{dl,B}}((\bm{L} \bm{n})\cdot\bm{t}, \jump{\bm{v}_B\cdot \bm{t}})_e+ (\bm{L},\nabla_h \bm{v}_B)_{\Omega_B}\\
&\;+ (\alpha\bm{u}_B,\bm{v}_B)_{\Omega_B}+\sum_{e\in \mathcal{F}_{pr,B}}(\jump{\bm{v}_B\cdot \bm{n}}, p_B)_e+\sum_{e\in \mathcal{F}_{h,\Gamma}}(\bm{v}_B\cdot \bm{n}, p_B)_e-(p_B,\tdiv_h  \bm{v}_B)_{\Omega_B}=(\bm{f}_B,\bm{v}_B)_{\Omega_B}.
\end{align*}
Multiplying \eqref{eq:Brinkman3} by a test function $q_B\in Q_h^0$ and performing integration by parts yield
\begin{align*}
(\nabla \cdot \bm{u}_B,q_B)_{\Omega_B} &=(\bm{u}_B\cdot \bm{n}_B, q_B)_{\Gamma_B}+\sum_{e\in \mathcal{F}_{h,\Gamma}}(\bm{u}_B\cdot \bm{n}_B, q_B)_e+\sum_{e\in \mathcal{F}_{dl,B}}(\bm{u}_B\cdot \bm{n}, \jump{q_B})_e-(\bm{u}_B,\nabla_h q_B)_{\Omega_B}\\
&=(g_1, q_B)_{\Gamma_B}-\sum_{e\in \mathcal{F}_{h,\Gamma}}(\bm{u}_D\cdot \bm{n}_D, q_B)_e+\sum_{e\in \mathcal{F}_{dl,B}}(\bm{u}_B\cdot \bm{n}, \jump{q_B})_e-(\bm{u}_B,\nabla_h q_B)_{\Omega_B}=0,
\end{align*}
where we use the interface condition $\bm{u}_B\cdot \bm{n}_B=\bm{u}_D\cdot \bm{n}_B$ (cf. \eqref{eq:interface1}) in the second equality.

Multiplying \eqref{eq:Darcy1} by a test function $\bm{v}_D\in H_{h,0}^D$ and performing integration by parts, we can obtain
\begin{align*}
(K_D^{-1}\bm{u}_D,\bm{v}_D)_{\Omega_D}+(\nabla p_D,\bm{v}_D)_{\Omega_D}&=(K_D^{-1}\bm{u}_D,\bm{v}_D)_{\Omega_D}+\sum_{e\in \mathcal{F}_{h,\Gamma}}(\bm{v}_D\cdot \bm{n}, p_D)_e-(p_D,\tdiv \bm{v}_D)_{\Omega_D}\\
&=(K_D^{-1}\bm{u}_D,\bm{v}_D)_{\Omega_D}+\sum_{e\in \mathcal{F}_{h,\Gamma}}(\bm{v}_D\cdot \bm{n}_D, p_B)_e-(p_D,\tdiv \bm{v}_D)_{\Omega_D}\\
&=(\bm{f}_D,\bm{v}_D)_{\Omega_D},
\end{align*}
where we use $p_B=p_D$ (cf. \eqref{eq:interface2}) in the second equality.

Based on the above derivations, we define the following bilinear forms for brevity
\begin{align*}
B_h^*(\bm{u}_{B,h},\bm{G}_B)&=\sum_{e\in \mathcal{F}_{dl,B}}(\jump{\bm{G}_B \bm{n}}, (\bm{u}_{B,h}\cdot \bm{n})\bm{n})_e-(\tdiv_h \bm{G}_B, \bm{u}_{B,h})_{\Omega_B},\\
B_h(\bm{L}_h,\bm{v}_B)&=- \sum_{e\in \mathcal{F}_{pr,B}^0}(\bm{L}_h \bm{n}, \jump{\bm{v}_B})_e-\sum_{e\in \mathcal{F}_{dl,B}}((\bm{L}_h \bm{n})\cdot\bm{t}, \jump{\bm{v}_B\cdot \bm{t}})_e+(\bm{L}_h,\nabla_h \bm{v}_B)_{\Omega_B},\\
b_h^*(p_{B,h},\bm{v}_B)&=\sum_{e\in \mathcal{F}_{pr,B}}(\jump{\bm{v}_B\cdot \bm{n}}, p_{B,h})_e+\sum_{e\in \mathcal{F}_{h,\Gamma}}(\bm{v}_B\cdot \bm{n}, p_{B,h})_e-(p_{B,h},\tdiv_h \bm{v}_B)_{\Omega_B},\\
b_h(\bm{u}_{B,h},q_B)&=-\sum_{e\in \mathcal{F}_{dl,B}}(\bm{u}_{B,h}\cdot \bm{n}, \jump{q_B})_e+(\bm{u}_{B,h},\nabla_h q_B)_{\Omega_B},\\
A_h(\bm{v}_D,p_{D,h})&=(p_{D,h},\nabla \cdot \bm{v}_D)_{\Omega_D},\\
I_h(p_{B,h},\bm{v}_D)&=\sum_{e\in \mathcal{F}_{h,\Gamma}}(\bm{v}_D\cdot \bm{n}_D,p_{B,h})_e.
\end{align*}


We are ready to propose the following discrete formulation for the coupled Brinkman-Darcy system \eqref{eq:Brinkman1}-\eqref{eq:interface2}: Find $(\bm{L}_{h},\bm{u}_{B,h},p_{B,h})\in W_h^B\times H_h^B\times Q_h^0$ and $(\bm{u}_{D,h},p_{D,h})\in H_{h,g_2}^D\times Q_h^D$ such that
\begin{align}
(\epsilon^{-1}\bm{L}_h,\bm{G}_B)_{\Omega_B}-B_h^*(\bm{u}_{B,h},\bm{G}_B)
+(K_D^{-1}\bm{u}_{D,h},\bm{v}_D)_{\Omega_D}&\nonumber\\
\;+I_h(p_{B,h},\bm{v}_D)-A_h(\bm{v}_D,p_{D,h})&=( \bm{f}_D,\bm{v}_D)_{\Omega_D},\label{eq:discreteB1}\\
B_h(\bm{L}_h,\bm{v}_B)+(\alpha\bm{u}_{B,h},\bm{v}_B)_{\Omega_B}
+b_h^*(p_{B,h},\bm{v}_B)+A_h(\bm{u}_{D,h},q_D)&=( \bm{f}_B,\bm{v}_B)_{\Omega_B}+(f,q_D)_{\Omega_D},\label{eq:discreteB2}\\
-I_h(q_B,\bm{u}_{D,h})-b_h(\bm{u}_{B,h},q_B)&=-(g_1,q_B)_{\Gamma_B}\label{eq:discreteB3}
\end{align}
for all $(\bm{G}_B,\bm{v}_{B},q_{B})\in W_h^B\times H_h^B\times Q_h^0$ and $(\bm{v}_{D},q_{D})\in H_{h,0}^D\times Q_h^D$. Hereafter, $\bm{u}_h$ is the velocity field defined by $\bm{u}_h\vert_{\Omega_i}=\bm{u}_{i,h},i=B,D$.
%

We introduce some properties that will be used later. First, integration by parts implies the following adjoint properties
\begin{alignat}{2}
B_h(\bm{G}_B,\bm{v}_B)&=B_h^*(\bm{v}_B,\bm{G}_B)&&\quad \forall (\bm{G}_B,\bm{v}_B)\in W_h^B\times H_h^B,\label{eq:adjoint1}\\
b_h(\bm{v}_B,q_B)&=b_h^*(q_B,\bm{v}_B)&&\quad \forall (\bm{v}_B, q_B)\in H_h^B\times Q_h^B.\label{eq:adjoint2}
\end{alignat}
Following \cite{ChungEngquist09}, we have the following inf-sup condition
\begin{align}
\|q_B\|_h\leq C_{\textnormal{inf}}\sup_{\bm{v}_{B}\in H_h^B}\frac{b_h(\bm{v}_{B},q_B)}{\|\bm{v}_{B}\|_{0,\Omega_B}}\quad \forall q_B\in Q_h^0.\label{eq:inf-sup}
\end{align}

Next, we introduce some interpolation operators that will be useful for the convergence analysis.
We define the projection operator $\Pi^{\text{BDM}}: H(\text{div};\Omega_D)\cap L^p(\Omega_D)\rightarrow H_h^D, p>2$ by following \cite{Brezzi85}
\begin{align*}
((\bm{v}-\Pi^{\text{BDM}}\bm{v})\cdot \bm{n},p_k)_e&=0\quad \forall p_k\in P^k(e) \;\text{and each edge}\; e\subset \partial \tau,\tau\in \mathcal{T}_{h,D},\\
(\bm{v}-\Pi^{\text{BDM}}\bm{v},\nabla p_{k-1})_\tau&=0\quad \forall p_{k-1}\in P^{k-1}(\tau),\tau\in \mathcal{T}_{h,D},\\
(\bm{v}-\Pi^{\text{BDM}}\bm{v},\mbox{curl}\, b)_\tau&=0\quad \forall b\in B^{k+1}(\tau),\tau\in \mathcal{T}_{h,D},
\end{align*}
where $B^{k+1}(\tau)=\{p\in P^{k+1}(\tau);\; p\mid_{\partial \tau}=0\}=\lambda_1\lambda_2\lambda_3 P^{k-2}(\tau)$. Here $\lambda_i,\;i=1,2,3$ are the barycentric coordinates of $\tau$.

It satisfies the following commutative properties
\begin{align*}
\nabla \cdot\Pi^{\text{BDM}}=\mathbb{P}_h\nabla \cdot,
\end{align*}
where $\mathbb{P}_h$ is the $L^2$-orthogonal projection onto $Q_h^D$.
The following convergence error estimates hold (see, e.g., \cite{Brezzi85,Duran88})
\begin{align}
\|\bm{v}-\Pi^{\text{BDM}}\bm{v}\|_{0,\Omega_D}&\leq C h^{k+1}\|\bm{v}\|_{k+1,\Omega_D}\quad \forall \bm{v}\in H^{k+1}(\Omega_D)^2,\label{eq:BDM}\\
\|\nabla \cdot (\bm{v}-\Pi^{\text{BDM}}\bm{v})\|_{0,\Omega_D}&\leq C h^{k}\|\nabla \cdot\bm{v}\|_{k,\Omega_D}\quad\; \forall \bm{v}\in H^{k}(\text{div};\Omega_D),\label{eq:divBDM}\\
\|q-\mathbb{P}_h q\|_{0,\Omega_D}&\leq C h^{k+1}\|q\|_{k+1,\Omega_D}\quad \forall q\in H^{k+1}(\Omega_D)\label{eq:Ph}.
\end{align}
In addition, we define a projection operator $\Pi_h$ for $W_h^B$ following \cite{ZhaoChungLam20}, which satisfies
\begin{align}
B_h(\Pi_h \bm{L}-\bm{L},\bm{v})=0\quad \forall \bm{v}\in H_h^B\label{eq:Bh1}
\end{align}
and the following interpolation error estimate holds
\begin{align}
\|\bm{L}-\Pi_h\bm{L}\|_{0,\Omega_B}&\leq C h^{k+1}|\bm{L}|_{k+1,\Omega_B}.\label{eq:Lerror}
\end{align}

To facilitate later analysis, we also define the following two projection operators (cf. \cite{LinaParkhybrid20}).
Let $I_h: H^1(\Omega_B)\rightarrow Q_h^B$ be defined by
\begin{equation}
\begin{split}
	(I_h q-q,\phi)_e
		&=0 \quad \forall \phi\in P^k(e),\forall e\in \mathcal{F}_{pr,B}\cup \mathcal{F}_{h,\Gamma},\\
	(I_hq-q,\phi)_\tau
		&=0\quad \forall \phi\in P^{k-1}(\tau),\forall \tau\in \mathcal{T}_{h,B}
\end{split}
\label{eq:Ih}
\end{equation}
and $J_h: L^2(\Omega_B)^2\cap H^{1/2+\delta}(\Omega_B)^2\rightarrow H_h^B$, $\delta>0$ be defined by
\begin{equation}
\begin{split}
	((J_h\bm{v}-\bm{v})\cdot\bm{n},\varphi)_e
		&=0\quad \forall \varphi\in P^{k}(e),\forall e\in \mathcal{F}_{dl,B},\\
	(J_h\bm{v}-\bm{v}, \bm{\phi})_\tau
		&=0\quad \forall \bm{\phi}\in P^{k-1}(\tau)^2, \forall \tau\in \mathcal{T}_{h,B}.
\end{split}
\label{eq:Jh}
\end{equation}
It is easy to see that $I_h$ and $J_h$ are well defined polynomial preserving operators. In addition, the following approximation properties hold for $q\in H^{k+1}(\Omega_B)$ and $\bm{v}\in H^{k+1}(\Omega_B)^2$ (cf. \cite{Ciarlet78,ChungEngquist09})
\begin{align}
\|q-I_hq\|_{0,\Omega_B}&\leq C h^{k+1}|q|_{k+1,\Omega_B},\label{eq:Perror}\\
\|\bm{v}-J_h\bm{v}\|_{0,\Omega_B}&\leq C h^{k+1}|\bm{v}|_{k+1,\Omega_B}\label{eq:uerror}.
\end{align}
By the definitions of $I_h$ and $J_h$, it readily holds
\begin{align}
b_h^*(p_B-I_hp_B,\bm{v})&=0\quad \forall \bm{v}\in H_h^B,\label{eq:bh1}\\
b_h(\bm{u}_B-J_h\bm{u}_B,q)&=0 \quad \forall q\in Q_h^B,\label{eq:bh2}\\
B_h^*(\bm{u}_B-J_h\bm{u}_B,\bm{G})&=0\quad \forall \bm{G}\in W_h^B.\label{eq:Bh2}
\end{align}
Following \cite{ZhaoChungLam20}, we have for any $\bm{v}\in H^1(\Omega_B)^2$
\begin{align}
\|J_h\bm{v}\|_{Z}&\leq C\|\bm{v}\|_{1,\Omega_B},\label{eq:Jhv}\\
\|J_h\bm{v}\|_{0,\Omega_B}&\leq C\|\bm{v}\|_{1,\Omega_B}.\label{eq:Jhv2}
\end{align}
To verify the mass conservation of the proposed scheme, we first show that $\int_{ \Gamma}\bm{u}_{D,h}\cdot\bm{n}_D\;ds=\int_{ \Gamma}\bm{u}_{D}\cdot\bm{n}_D\;ds$. Note that
\begin{align*}
\int_{\Omega_D}\nabla\cdot\bm{u}_{D,h}\;dx=\int_{\Omega_D}f\;dx,
\end{align*}
which yields
\begin{align*}
\int_{ \Gamma_D}\bm{u}_{D,h}\cdot\bm{n}_D\;ds+\int_{ \Gamma}\bm{u}_{D,h}\cdot\bm{n}_D\;ds=\int_{\Omega_D}f\;dx.
\end{align*}
Hence,
\begin{align*}
\int_{ \Gamma_D}\bm{u}_{D,h}\cdot\bm{n}_D\;ds+\int_{ \Gamma}\bm{u}_{D,h}\cdot\bm{n}_D\;ds=\int_{ \Gamma_D}g_2\;ds+\int_{ \Gamma}\bm{u}_{D}\cdot\bm{n}_D\;ds,
\end{align*}
which implies
\begin{align}
\int_{ \Gamma}\bm{u}_{D,h}\cdot\bm{n}_D\;ds=\int_{ \Gamma}\bm{u}_{D}\cdot\bm{n}_D\;ds.\label{eq:uDh}
\end{align}
We remark that the property \eqref{eq:uDh} is crucial for the proof of the mass conservation. For simplicity, we assume that $f$ is a polynomial function hereafter and belongs to $Q_h^D$.

\begin{lemma}(strong mass conservation).
The interface condition \eqref{eq:interface1} is satisfied exactly for the discrete solution, i.e.,
\begin{align*}
\bm{u}_{B,h}\cdot\bm{n}_D=\bm{u}_{D,h}\cdot\bm{n}_D \quad\textnormal{on}\;\Gamma.
\end{align*}
In addition, $\bm{u}_{h}\in H(\textnormal{div};\Omega)$, $\nabla \cdot \bm{u}_{B,h}=0$ and $\nabla\cdot\bm{u}_{D,h}=f$. It holds
\begin{align*}
\bm{u}_{B,h}\cdot\bm{n}_B=\Pi_{h,\Gamma_B}g_1\quad \textnormal{on}\;\Gamma_B
\end{align*}
and
\begin{align}
\nabla \cdot (\bm{u}-\bm{u}_h)=0\quad \textnormal{in}\;\Omega\label{eq:massconservation}.
\end{align}

\end{lemma}

\begin{proof}
First, note that \eqref{eq:discreteB3} holds for any $q_B\in Q_h^B$. Indeed, we have from \eqref{eq:uDg1} and \eqref{eq:uDh}
\begin{align*}
\int_{\Gamma}\bm{u}_{D,h}\cdot\bm{n}_D\;ds=\int_{\Gamma}\bm{u}_{D}\cdot\bm{n}_D\;ds=\int_{\Gamma_B} g_1\;ds,
\end{align*}
which implies that \eqref{eq:discreteB3} holds for any $q_B=c$, where $c$ is a constant.

From \eqref{eq:discreteB3} and the adjoint property \eqref{eq:adjoint2}, we can infer that
\begin{equation}
\begin{split}
&-\sum_{e\in \mathcal{F}_{h,\Gamma}}(\bm{u}_{D,h}\cdot \bm{n}_D,q_{B})_e-\sum_{e\in \mathcal{F}_{pr,B}}(\jump{\bm{u}_{B,h}\cdot \bm{n}}, q_{B})_e-\sum_{e\in \mathcal{F}_{h,\Gamma}}(\bm{u}_{B,h}\cdot \bm{n}_B, q_{B})_e\\
&\;+\sum_{\tau\in \mathcal{T}_{h,B}}(q_{B},\nabla \cdot \bm{u}_{B,h})_{\tau}=-(\Pi_{h,\Gamma_B}g_1,q_B)_{\Gamma_B}\quad \forall q_B\in Q_h^B.
\end{split}
\label{eq:conservation}
\end{equation}
We can take $q_B$ in line with \eqref{eq:dof1}-\eqref{eq:dof2} such that
\begin{alignat*}{2}
(q_B, p_k)_e&=-(\bm{u}_{B,h}\cdot\bm{n}_B-\Pi_{h,\Gamma_B}g_1,p_k)_e&&\quad \forall p_k\in P^k(e),e\in \Gamma_B,\\
(q_B, p_k)_e&=-(\jump{\bm{u}_{B,h}\cdot\bm{n}},p_k)_e&&\quad \forall p_k\in P^k(e),e\in \mathcal{F}_{pr,B}\backslash\Gamma_B,\\
(q_B,p_k)_e&=(\bm{u}_{B,h}\cdot\bm{n}_D-\bm{u}_{D,h}\cdot\bm{n}_D,p_k)_e&&\quad \forall p_k\in P^k(e),e\in \mathcal{F}_{h,\Gamma},\\
(q_B,p_{k-1})_\tau&=(\nabla\cdot\bm{u}_{B,h},p_{k-1})_\tau&&\quad \forall p_{k-1}\in P^{k-1}(\tau),\tau\in \mathcal{T}_{h,B}.
\end{alignat*}
Then we can infer from \eqref{eq:conservation} that
\begin{align*}
&\sum_{e\in\mathcal{F}_{h,\Gamma}}\|\bm{u}_{B,h}\cdot\bm{n}_D-\bm{u}_{D,h}\cdot\bm{n}_D\|_{0,e}^2+\sum_{e\in \mathcal{F}_{pr,B}^0}\|\jump{\bm{u}_{B,h}\cdot\bm{n}}\|_{0,e}^2+\sum_{\tau\in \mathcal{T}_{h,B}}\|\nabla\cdot\bm{u}_{B,h}\|_{0,\tau}^2\\
&+\sum_{e\in \Gamma_B}\|\bm{u}_{B,h}\cdot\bm{n}_B-\Pi_{h,\Gamma_B}g_1\|_{0,e}^2=0,
\end{align*}
which yields
\begin{alignat*}{2}
\bm{u}_{B,h}\cdot\bm{n}_D&=\bm{u}_{D,h}\cdot\bm{n}_D &&\quad\mbox{on}\;\Gamma,\\
\jump{\bm{u}_{B,h}\cdot\bm{n}}_e&=0&&\quad \forall e\in \mathcal{F}_{pr,B}\backslash\Gamma_B,\\
\nabla\cdot\bm{u}_{B,h}\mid_\tau&=0&&\quad\forall \tau\in \mathcal{T}_{h,B},\\
\bm{u}_{B,h}\cdot\bm{n}_B\mid_e&=\Pi_{h,\Gamma_B}g_1\mid_e&&\quad\forall e\in \Gamma_B.
\end{alignat*}

Hence, $\bm{u}_{B,h}$ is divergence free in $\Omega_B$, the interface condition \eqref{eq:interface1} is satisfied exactly and $\bm{u}_{B,h}\cdot\bm{n}=\Pi_{h,\Gamma_B}g_1$ on $\Gamma_B$.
Finally, taking $q_D=\nabla\cdot\bm{u}_{D,h}- f$ and $\bm{v}_B=\bm{0}$ in \eqref{eq:discreteB2} implies that $\nabla \cdot \bm{u}_{D,h}= f$. Thus, \eqref{eq:massconservation} holds.

\end{proof}

\begin{remark}
Thanks to the special choice of the finite element spaces, we are able to achieve a $H(\tdiv;\Omega)$-conforming velocity over the whole domain. This choice also benefits the treatment of interface conditions. Indeed, the interface conditions can be imposed exactly without resorting to additional variables. Importantly, the proposed scheme satisfies the mass conservation exactly. These desirable merits make our scheme a good candidate for the simulation of the coupled flow and transport. For the sake of simplicity, we adopt the interface conditions \eqref{eq:interface1}-\eqref{eq:interface2}, and our scheme can also be extended to solve the coupling of Brinkman-Darcy flow with Beavers-Joseph-Saffman interface conditions.

\end{remark}

Since $\bm{u}_h$ is the $L^2$-orthogonal projection of $\bm{u}$ on the boundary, we have from the approximation properties of $\Pi_{h,\Gamma_i}, i=B,D$ that
\begin{align}
\|(\bm{u}-\bm{u}_h)\cdot\bm{n}\|_{0,\partial \Omega}\leq Ch^{k+1}\|\bm{u}\|_{k+1,\partial \Omega}.\label{eq:interpGin}
\end{align}

\begin{theorem}(unique solvability).
There exists a unique solution to \eqref{eq:discreteB1}-\eqref{eq:discreteB3}.

\end{theorem}

\begin{proof}

As \eqref{eq:discreteB1}-\eqref{eq:discreteB3} is a square linear system, uniqueness implies existence. Thus, it suffices to show the uniqueness. To this end, we set $\bm{f}_D=\bm{f}_B=\bm{0}$ and $f=g_1=0$.
Then taking $\bm{G}_B=\bm{L}_h$, $\bm{v}_B=\bm{u}_{B,h}$, $q_B=p_{B,h}$, $\bm{v}_D=\bm{u}_{D,h}$ and $q_D=p_{D,h}$ in \eqref{eq:discreteB1}-\eqref{eq:discreteB3} and summing up the resulting equations, we can obtain
\begin{align*}
\|\epsilon^{-\frac{1}{2}}\bm{L}_h\|_{0,\Omega_B}^2+\|K_D^{-\frac{1}{2}}\bm{u}_{D,h}\|_{0,\Omega_D}^2+ \|\alpha^{\frac{1}{2}}\bm{u}_{B,h}\|_{0,\Omega_B}^2=0.
\end{align*}
Thus, we can infer that $\bm{L}_h=\bm{0}$ and $\bm{u}_{D,h}=\bm{u}_{B,h}=\bm{0}$.

On the other hand, we have from \eqref{eq:discreteB2}, the inf-sup condition \eqref{eq:inf-sup} and the adjoint property \eqref{eq:adjoint2} that
\begin{align*}
\|p_{B,h}\|_h\leq C \sup_{\bm{v}_B\in H_h^B}\frac{b_h(\bm{v}_B,p_{B,h})}{\|\bm{v}_B\|_{0,\Omega_B}}=C\sup_{\bm{v}_B\in H_h^B}\frac{b_h^*(p_{B,h},\bm{v}_B)}{\|\bm{v}_B\|_{0,\Omega_B}}=0.
\end{align*}
Since $\|p_{B,h}\|_h$ defines a norm on $Q_h^0$, it follows that $p_{B,h}=0$.

Finally, we have from the inf-sup condition (cf. \cite{Raviart77}) and \eqref{eq:discreteB1} that
\begin{align*}
\|p_{D,h}\|_{0,\Omega_D}\leq C \sup_{\bm{v}_D\in H_h^D}\frac{(p_{D,h},\nabla\cdot \bm{v}_D)_{\Omega_D}}{\|\bm{v}_D\|_{\text{div},\Omega_D}}=0.
\end{align*}
Hence $p_{D,h}=0$. Therefore, the proof is completed.
\end{proof}

\section{A priori error estimate}\label{sec:error}

In this section, we will prove the convergence error estimates for all the variables measured in proper norms. In particular, the velocity error is shown to be independent of the pressure variable. To this end, we first prove the following inf-sup condition, which will play an important role for later analysis.

\begin{lemma}\label{lemma:inf-supD}
There exists a positive constant $C$ independent of the meshsize such that
\begin{align*}
\|q\|_{0,\Omega_D}\leq C \sup_{\bm{v}\in H_{0,\Gamma}(\textnormal{div};\Omega_D)}\frac{(\nabla \cdot \bm{v},q)_{\Omega_D}}{\|\bm{v}\|_{\textnormal{div},\Omega_D}}
\end{align*}
for any $q\in L^2(\Omega_D)$.

\end{lemma}

\begin{proof}

Consider the boundary value problem
\begin{align*}
\Delta z&=q\quad\mbox{in}\;\Omega_D,\\
z&=0\quad\mbox{on}\;\Gamma_D,\\
\nabla z\cdot\bm{n}&=0\quad\mbox{on}\;\Gamma.
\end{align*}
The weak formulation reads: Find $z\in H^1_{0,\Gamma_D}(\Omega_D)$ such that
\begin{align}
(\nabla z, \nabla w)_{\Omega_D}=-(q,w)_{\Omega_D}.\label{eq:weak}
\end{align}

The Lax-Milgram lemma implies that \eqref{eq:weak} has a unique solution $z\in H^1_{0,\Gamma_D}(\Omega_D)$. In addition, we have $\|z\|_{2,\Omega_D}\leq C\|q\|_{0,\Omega_D}$.

Let $\hat{\bm{\sigma}}:=\nabla z$, we have $\nabla \cdot \hat{\bm{\sigma}} =q$ in $\Omega_D$, in addition $\hat{\bm{\sigma}}\cdot \bm{n}=0$ on $\Gamma$, which yields $\hat{\bm{\sigma}}\in H_{0,\Gamma}(\text{div};\Omega_D)$. Then we have
\begin{align*}
\sup_{\hat{\bm{\sigma}}\in H_{0,\Gamma}(\text{div};\Omega_D)}\frac{(\nabla \cdot \hat{\bm{\sigma}},q)_{\Omega_D}}{\|\hat{\bm{\sigma}}\|_{\text{div},\Omega_D}} \geq \frac{\|q\|_{0,\Omega_D}^2}{\|\hat{\bm{\sigma}}\|_{\text{div},\Omega_D}}\geq C \|q\|_{0,\Omega_D}.
\end{align*}

\end{proof}


The following error equations can be easily obtained by performing integration by parts on the discrete formulation \eqref{eq:discreteB1}-\eqref{eq:discreteB3}
\begin{align}
&(\epsilon^{-1}(\bm{L}-\bm{L}_{h}),\bm{G}_B)_{\Omega_B}-B_h^*(\bm{u}_B-\bm{u}_{B,h},\bm{G}_B)
+(K_D^{-1}(\bm{u}_D-\bm{u}_{D,h}),\bm{v}_D)_{\Omega_D}\nonumber\\
&+I_h(p_B-p_{B,h},\bm{v}_D)-A_h(\bm{v}_D,p_D-p_{D,h})=0,\label{eq:errorf1}\\
&B_h(\bm{L}-\bm{L}_h,\bm{v}_B)+(\alpha(\bm{u}_B-\bm{u}_{B,h}),\bm{v}_B)_{\Omega_B}+
b_h^*(p_B-p_{B,h},\bm{v}_B)+A_h(\bm{u}_D-\bm{u}_{D,h},q_D)=0,\label{eq:errorf2}\\
&-I_h(q_{B},\bm{u}_D-\bm{u}_{D,h})-b_h(\bm{u}_B-\bm{u}_{B,h},q_B)=0\label{eq:errorf3}
\end{align}
for all $(\bm{G}_B,\bm{v}_B, q_{B})\in W_h^B\times H_h^B\times Q_h^0$ and $(\bm{v}_D, q_D)\in H_{h,0}^D\times Q_h^D$.

\begin{lemma}\label{lemma:L2f}

Let $ (\bm{L}_h,\bm{u}_{B,h},p_{B,h})\in W_h^B\times H_h^B\times Q_h^0$ and $(\bm{u}_{D,h},p_{D,h})\in H_{h,g_2}^D\times Q_h^D$ be the discrete solution of \eqref{eq:discreteB1}-\eqref{eq:discreteB3}. Then, there exists a positive constant $C$ independent of the mesh size such that
\begin{align*}
&\|\epsilon^{-\frac{1}{2}}(\Pi_h \bm{L}-\bm{L}_h)\|_{0,\Omega_B}+\|K_D^{-\frac{1}{2}}(\Pi^{\textnormal{BDM}}\bm{u}_D-\bm{u}_{D,h})\|_{0,\Omega_D}+
\|\alpha^{\frac{1}{2}}(J_h\bm{u}_B-\bm{u}_{B,h})\|_{0,\Omega_B}\\
&\leq C \Big(\|\epsilon^{-\frac{1}{2}}(\Pi_h \bm{L}-\bm{L})\|_{0,\Omega_B}+\|K_D^{-\frac{1}{2}}(\Pi^{\textnormal{BDM}}\bm{u}_D-\bm{u}_{D})\|_{0,\Omega_D}+ \|\alpha^{\frac{1}{2}} (J_h\bm{u}_B-\bm{u}_{B})\|_{0,\Omega_B}\Big).
\end{align*}

\end{lemma}

\begin{proof}

Taking $\bm{G}_B=\Pi_h \bm{G}-\bm{G}_h$, $\bm{v}_B=J_h\bm{u}_B-\bm{u}_{B,h}$, $q_B = I_hp_B-p_{B,h}$, $\bm{v}_D=\Pi^{\text{BDM}}\bm{u}_D-\bm{u}_{D,h}$, $q_D=\mathbb{P}_h p_D-p_{D,h}$ in \eqref{eq:errorf1}-\eqref{eq:errorf3} and summing up the resulting equations yield
\begin{equation}
\begin{split}
&(\epsilon^{-1}(\bm{L}-\bm{L}_h),\Pi_h \bm{L}-\bm{L}_h)_{\Omega_B}+(K_D^{-1}(\bm{u}_D-\bm{u}_{D,h}),\Pi^{\text{BDM}}\bm{u}_D-\bm{u}_{D,h})_{\Omega_D}\\
&\;+(\alpha(\bm{u}_B-\bm{u}_{B,h}),J_h\bm{u}_B-\bm{u}_{B,h})_{\Omega_B}+\sum_{e\in \mathcal{F}_{h,\Gamma}}((\Pi^{\text{BDM}}\bm{u}_D-\bm{u}_{D,h})\cdot \bm{n}_D,p_B-p_{B,h})_e\\
&\;-\sum_{e\in  \mathcal{F}_{h,\Gamma}}((\bm{u}_D-\bm{u}_{D,h})\cdot \bm{n}_D,I_hp_B-p_{B,h})_e=0,
\end{split}
\label{eq:errorn}
\end{equation}
where we employ \eqref{eq:adjoint1}, \eqref{eq:adjoint2}, \eqref{eq:Bh1} and \eqref{eq:bh1}-\eqref{eq:Bh2}.

%
%
%
%
%
We can infer from the definitions of $I_h$ and $\Pi^{\text{BDM}}$ that
\begin{align*}
\sum_{e\in \mathcal{F}_{h,\Gamma}}((\Pi^{\text{BDM}}\bm{u}_D-\bm{u}_{D,h})\cdot \bm{n}_D,p_B-p_{B,h})_e-\sum_{e\in \mathcal{F}_{h,\Gamma}}((\bm{u}_D-\bm{u}_{D,h})\cdot \bm{n}_D,I_hp_B-p_{B,h})_e=0.
\end{align*}
Therefore, \eqref{eq:errorn} can be rewritten as
\begin{align*}
&(\epsilon^{-1}(\bm{L}-\bm{L}_h),\Pi_h \bm{L}-\bm{L}_h)_{\Omega_B}+(K_D^{-1}(\bm{u}_D-\bm{u}_{D,h}),\Pi^{\text{BDM}}\bm{u}_D-\bm{u}_{D,h})_{\Omega_D}\\
&\;+(\alpha(\bm{u}_B-\bm{u}_{B,h}),J_h\bm{u}_B-\bm{u}_{B,h})_{\Omega_B}=0,
\end{align*}
which coupled with Young's inequality leads to
\begin{align*}
&\|\epsilon^{-\frac{1}{2}}(\Pi_h \bm{L}-\bm{L}_h)\|_{0,\Omega_B}^2+\|K_D^{-\frac{1}{2}}(\Pi^{\text{BDM}}\bm{u}_D-\bm{u}_{D,h})\|_{0,\Omega_D}^2+
\|\alpha^{\frac{1}{2}}(J_h\bm{u}_B-\bm{u}_{B,h})\|_{0,\Omega_B}^2\\
&\;\leq C \Big(\|\epsilon^{-\frac{1}{2}}(\Pi_h \bm{L}-\bm{L})\|_{0,\Omega_B}^2
+\|K_D^{-\frac{1}{2}}(\Pi^{\text{BDM}}\bm{u}_D-\bm{u}_{D})\|_{0,\Omega_D}^2+\|\alpha^{\frac{1}{2}}(J_h\bm{u}_B-\bm{u}_{B})\|_{0,\Omega_B}^2\Big).
\end{align*}
Therefore, the proof is completed.

\end{proof}

\begin{lemma}\label{lemma:p}
Let $ (\bm{L}_h,\bm{u}_{B,h},p_{B,h})\in W_h^B\times H_h^B\times Q_h^0$ and $(\bm{u}_{D,h},p_{D,h})\in H_{h,g_2}^D\times Q_h^D$ be the discrete solution of \eqref{eq:discreteB1}-\eqref{eq:discreteB3}.  Then, the following convergence estimates hold
\begin{align*}
\|\mathbb{P}_h p_D-p_{D,h}\|_{0,\Omega_D}&\leq C \|\bm{u}_D-\bm{u}_{D,h}\|_{0,\Omega_D},\\
\|I_hp_B-p_{B,h}\|_{0,\Omega_B}&\leq C  \Big( \|\Pi_h\bm{L}-\bm{L}_h\|_{0,\Omega_B}+\alpha_{\textnormal{max}}^{\frac{1}{2}}\|\alpha^{\frac{1}{2}}(\bm{u}_B-\bm{u}_{B,h})\|_{0,\Omega_B}\Big),\\
\|\nabla \cdot (\bm{u}_D-\bm{u}_{D,h})\|_{0,\Omega_D}&=\|\nabla \cdot (\Pi^{\textnormal{BDM}}\bm{u}_D-\bm{u}_{D})\|_{0,\Omega_D},
\end{align*}
where $\alpha_{\textnormal{max}}$ is the maximum eigenvalue of $\alpha$.

\end{lemma}

\begin{proof}
Note that $p_{B,h}\in L^2_0(\Omega_B)$ and $I_hp_B\in L^2_0(\Omega_B)$ (cf. \eqref{eq:Ih}), thereby $I_hp_B-p_{B,h}\in L^2_0(\Omega_B)$, then it is well known that the following inf-sup condition holds (cf. \cite{GiraultRaviart86})
\begin{align*}
\|I_hp_B-p_{B,h}\|_{0,\Omega_B}\leq C \sup_{\bm{v}\in H^1_{0}(\Omega_B)^2}\frac{(\nabla \cdot \bm{v},I_hp_B-p_{B,h})_{\Omega_B}}{\|\bm{v}\|_{1,\Omega_B}},
\end{align*}
where we can estimate the numerator of the right-hand side by the Cauchy-Schwarz inequality, \eqref{eq:Bh1}, \eqref{eq:bh2} and \eqref{eq:errorf2}
\begin{align*}
&(\nabla \cdot \bm{v},I_hp_B-p_{B,h})_{\Omega_B}=\sum_{e\in \mathcal{F}_{dl,B}}(\bm{v}\cdot \bm{n},\jump{I_hp_B-p_{B,h}})_e-\sum_{\tau\in \mathcal{T}_{h,B}}(\bm{v},\nabla (I_hp_B-p_{B,h}))_\tau\\
&=-b_h(\bm{v},I_hp_B-p_{B,h})=-b_h(J_h\bm{v},I_hp_B-p_{B,h})\\
&=B_h(\Pi_h\bm{L}-\bm{L}_h,J_h\bm{v})+ (\alpha(\bm{u}_B-\bm{u}_{B,h}),J_h\bm{v})_{\Omega_B}\\
&\leq C \Big(\|\Pi_h\bm{L}-\bm{L}_h\|_{0,\Omega_B}\|J_h\bm{v}\|_{Z}+\alpha_{\text{max}}^{\frac{1}{2}} \|\alpha^{\frac{1}{2}}(\bm{u}_B-\bm{u}_{B,h})\|_{0,\Omega_B}\|J_h\bm{v}\|_{0,\Omega_B}\Big).
\end{align*}
The above estimates coupled with \eqref{eq:Jhv} and \eqref{eq:Jhv2} lead to
\begin{align*}
\|I_hp_B-p_{B,h}\|_{0,\Omega_B}\leq C \Big(\|\Pi_h\bm{L}-\bm{L}_h\|_{0,\Omega_B}+
\alpha_{\text{max}}^{\frac{1}{2}}\|\alpha^{\frac{1}{2}}(\bm{u}_B-\bm{u}_{B,h})\|_{0,\Omega_B}\Big).
\end{align*}
Next, we prove the error estimate for $\|\mathbb{P}_hp_D-p_{D,h}\|_{0,\Omega_D}$.
We have from Lemma~\ref{lemma:inf-supD} that
\begin{equation}
\begin{split}
\|\mathbb{P}_hp_D-p_{D,h}\|_{0,\Omega_D}&\leq C \sup_{\bm{v}\in H_{0,\Gamma}(\text{div};\Omega_D)}\frac{(\mathbb{P}_hp_D-p_{D,h},\nabla\cdot \bm{v})_{\Omega_D}}{\|\bm{v}\|_{\text{div},\Omega_D}}\\
&= C \sup_{\bm{v}\in H_{0,\Gamma}(\text{div};\Omega_D)}\frac{A_h(\bm{v},\mathbb{P}_hp_D-p_{D,h})}{\|\bm{v}\|_{\text{div},\Omega_D}}.
\end{split}
\label{eq:pD2}
\end{equation}
Following \cite[Section~4.2]{Gatica14}, we can infer that there exists a Fortin interpolation operator $\Pi^F$ which satisfies
\begin{align*}
A_h(\bm{v}-\Pi^F \bm{v},q)=0\quad \forall q\in Q_h^D,\\
\|\Pi^F\bm{v}\|_{\text{div},\Omega_D}\leq C \|\bm{v}\|_{\text{div},\Omega_D}.
\end{align*}
Therefore, we have from \eqref{eq:errorf1} and \eqref{eq:pD2} that
\begin{align*}
\|\mathbb{P}_hp_D-p_{D,h}\|_{0,\Omega_D}& \leq C \sup_{\bm{v}\in H_{0,\Gamma}(\text{div};\Omega_D)}
\frac{A_h(\Pi^F\bm{v},\mathbb{P}_hp_D-p_{D,h})}{\|\bm{v}\|_{\text{div},\Omega_D}}\\
&=C \sup_{\bm{v}\in H_{0,\Gamma}(\text{div};\Omega_D)}\frac{(K_D^{-1}(\bm{u}_D-\bm{u}_{D,h}),\Pi^F\bm{v})_{\Omega_D}}{\|\bm{v}\|_{\text{div},\Omega_D}}\\
&\leq C \|K_D^{-1}(\bm{u}_D-\bm{u}_{D,h})\|_{0,\Omega_D}.
\end{align*}
On the other hand, we can deduce from \eqref{eq:errorf2} that
\begin{align*}
(\nabla \cdot (\Pi^{\text{BDM}}\bm{u}_D-\bm{u}_{D,h}),q_D)_{\Omega_D}=0\quad \forall q_D\in Q_h^D,
\end{align*}
which implies $\nabla \cdot (\Pi^{\text{BDM}}\bm{u}_D-\bm{u}_{D,h})\mid_{\Omega_D}=0$. Therefore, we have
\begin{align*}
\|\nabla \cdot (\bm{u}_D-\bm{u}_{D,h})\|_{0,\Omega_D}=\|\nabla \cdot (\Pi^{\text{BDM}}\bm{u}_D-\bm{u}_{D})\|_{0,\Omega_D}.
\end{align*}

\end{proof}

Combining Lemmas~\ref{lemma:L2f} and \ref{lemma:p}, and the interpolation error estimates \eqref{eq:BDM}-\eqref{eq:Ph}, \eqref{eq:Lerror} and \eqref{eq:Perror}-\eqref{eq:uerror} leads to the next theorem.

\begin{theorem}\label{thm:L2}
Assume that $(\bm{L},\bm{u})\in H^{k+1}(\Omega_B)^{2\times 2} \times (H^{k+1}(\Omega)^2\cap H^k(\textnormal{div};\Omega_D))$ and $(p_B,p_D)\in H^{k+1}(\Omega_B)\times H^{k+1}(\Omega_D)$. Let $ (\bm{L}_h,\bm{u}_{B,h},p_{B,h})\in W_h^B\times H_h^B\times Q_h^0$ and $(\bm{u}_{D,h},p_{D,h})\in H_{h,g_2}^D\times Q_h^D$ be the discrete solution of \eqref{eq:discreteB1}-\eqref{eq:discreteB3}, then the following estimates hold
\begin{align*}
&\|\epsilon^{-\frac{1}{2}}( \bm{L}-\bm{L}_h)\|_{0,\Omega_B}+\|K_D^{-\frac{1}{2}}(\bm{u}_D-\bm{u}_{D,h})\|_{0,\Omega_D}+
\|\alpha^{\frac{1}{2}}(\bm{u}_B-\bm{u}_{B,h})\|_{0,\Omega_B}\\
&\leq C h^{k+1}\Big(\epsilon^{-\frac{1}{2}}\|\bm{L}\|_{k+1,\Omega_B}+\|\bm{u}\|_{k+1}\Big),\\
&\|p_D-p_{D,h}\|_{0,\Omega_D}\leq C h^{k+1}\Big(\epsilon^{-\frac{1}{2}}\|\bm{L}\|_{k+1,\Omega_B}+\|\bm{u}\|_{k+1}
+\|p_D\|_{k+1,\Omega_D}\Big),\\
&\|p_B-p_{B,h}\|_{0,\Omega_B}\leq C h^{k+1}\Big((1+\epsilon^{-\frac{1}{2}})\|\bm{L}\|_{k+1,\Omega_B}+\|\bm{u}\|_{k+1}+
\|p_B\|_{k+1,\Omega_B}\Big),\\
&\|\nabla \cdot (\bm{u}_D-\bm{u}_{D,h})\|_{0,\Omega_D}\leq C h^{k}\|\nabla\cdot\bm{u}_D\|_{k,\Omega_D}.
\end{align*}

\end{theorem}

\begin{remark}
We can observe from Theorem~\ref{thm:L2} that the convergence error estimate for velocity $\bm{u}_B$ is independent of the
pressure variable $p_B$, which demonstrates the pressure-robustness of the proposed scheme. The study of the pressure-robust schemes for incompressible flow is important from a practical point of view, see, e.g, \cite{FuQiu18,Frerichs20,Mu20,Wang21}.

\end{remark}

\section{Upwinding staggered DG method for transport equation}\label{sec:transport}
In this section, we devise a new staggered DG scheme for the transport equation, where the upwinding fluxes as well as the boundary correction terms are exploited to improve the performance of the scheme. For this purpose, we first define the spaces that will be used for the approximation of the transport equation:
\begin{align*}
& U_h := \{ \phi_h\vert_\tau \in P^k(\tau), \forall\tau \in \mathcal{T}_{h}; \jump{\phi_h}_e = 0, \forall e \in \mathcal{F}_{pr}^0\}, \\
& W_h:= \{ \bm{q}_h\vert_{\tau} \in P^k(\tau)^2,  \forall\tau \in \mathcal{T}_{h};  \jump{\bm{q}_h  \cdot \bm{n}}_e = 0, \forall e \in \mathcal{F}_{dl} \}.
\end{align*}
For any $\phi_h\in U_h$, we define
\begin{align*}
\|\phi_h\|_{1,h,*}^2:=\sum_{\tau\in \mathcal{T}_h}\|\nabla \phi_h\|_{0,\tau}^2+\sum_{e\in \mathcal{F}_{dl}}h_e^{-1}\|\jump{\phi_h}\|_{0,e}^2.
\end{align*}
Following \cite[Lemma~3.1]{Feng02}, we have the following trace inequality
\begin{align}
\|\phi_h\|_{0,\partial \Omega}\leq C \|\phi_h\|_{1,h}\quad \forall \phi_h\in U_h.\label{eq:Gammaz}
\end{align}
We rewrite the transport equation by introducing the diffusive flux
\begin{align}
\bm{z}=-K\nabla c.\label{eq:transport1}
\end{align}
Then we can recast the transport equation \eqref{eq:modeltransport} into the following first-order system
\begin{align}
\phi c_t+\nabla \cdot (c\bm{u}+\bm{z})&=\phi s+\hat{c}f^+ - cf^-\quad \mbox{in}\;\Omega,\label{eq:transport2}\\
(c\bm{u}+\bm{z})\cdot\bm{n}&=c_{\text{in}} \bm{u}\cdot\bm{n}\quad \mbox{on}\;\Gamma_{\text{in}},\\
\bm{z}\cdot\bm{n}&=0\quad \mbox{on}\;\Gamma_{\text{out}}\label{eq:transport4}.
\end{align}
Multiplying \eqref{eq:transport1} by a test function $\bm{\psi}\in W_h$ and performing integration by parts yield
\begin{equation}\label{eq:var-2}
(K^{-1} \bm{z}, \bm{\psi}) - \sum_{\tau \in \mathcal{T}_h}(c, \nabla\cdot \bm{\psi})_\tau + \sum_{e \in \mathcal{F}_{pr}} (c, \jump{\bm{\psi}} \cdot \bm{n}_e)_e  = 0.
\end{equation}
Then multiplying \eqref{eq:transport2} by a test function $q\in U_h$ and performing integration by parts imply that
\begin{equation}
\begin{split}
&(\phi \frac{\partial c}{\partial t},q)+(cf^-,q)- \sum_{\tau \in \mathcal{T}_h} (\bm{u} c+ \bm{z}, \nabla q)_\tau + \sum_{e \in \mathcal{F}_{dl}} \int_e\jump{\widetilde{c} q } \bm{u}\cdot \bm{n}_e~ds + \sum_{e \in \mathcal{F}_{dl}} \int_e \bm{z}\cdot \bm{n}_e\jump{q}  ~ds \\
&\; +
(\bm{u}\cdot\bm{n}c,q)_{\Gamma_{\text{out}}} = (\phi s, q)-
(c_{\text{in}}\bm{u}\cdot\bm{n},q)_{\Gamma_{\text{in}}}+(\hat{c}f^+,q),
\end{split}
\label{eq:derivation1}
\end{equation}
where we use the upwinding flux to define $\widetilde{c}$, namely
\begin{equation}\label{eq:up-I-c}
\widetilde{c}:= \left\{
\begin{aligned}
& c\mid_{\tau_1} \quad & & \text{ if } \bm{u} \cdot \bm{n}_e \ge 0 & & \quad \text{(outflow)}, \\
& c\mid_{\tau_2} \quad & & \text{ if } \bm{u} \cdot \bm{n}_e < 0 & & \quad \text{(inflow)}.
\end{aligned}
\right.
\end{equation}
Here $\tau_1$ and $\tau_2$ are the two triangles sharing the common edge $e$ and $\bm{n}_e$ points from $\tau_1$ to $\tau_2$.

Thereby, we can rewrite the second term in \eqref{eq:derivation1} as
\begin{align*}
\sum_{e \in \mathcal{F}_{dl}} \int_e\jump{\widetilde{c} q} \bm{u}\cdot \bm{n}~ds=\sum_{e \in \mathcal{F}_{dl}}\int_e\avg{c}\jump{q} \bm{u}\cdot \bm{n}~ds+\frac{1}{2}\sum_{e \in \mathcal{F}_{dl}}\int_e\jump{c}\jump{q} |\bm{u}\cdot \bm{n}|~ds.
\end{align*}
Then the discrete formulation for \eqref{eq:transport1}-\eqref{eq:transport4} reads as follows: Find $(\bm{z}_h,c_h)\in W_h\times U_h$ such that
\begin{align}
&(K^{-1} \bm{z}_h, \bm{\psi}) - T_h^*(c_h,\bm{\psi})=0,\label{eq:discrete1}\\
&(\phi \frac{\partial c_h}{\partial t},q)+(c_hf^-,q)+T_h(\bm{z}_h,q)-(\bm{u}_hc_h,\nabla q)+S_h(c_h,q)+
(\bm{u}_hc_h\cdot\bm{n},q)_{\Gamma_{\text{out}}}\nonumber\\
&\;+\frac{1}{2}((\bm{u}-\bm{u}_h)\cdot\bm{n}c_h,q)_{\Gamma_{\text{out}}}
-\frac{1}{2}((\bm{u}-\bm{u}_h)\cdot\bm{n}c_h,q)_{\Gamma_{\text{in}}} = (\phi s, q)-
(c_{\text{in}}\bm{u}\cdot\bm{n},q)_{\Gamma_{\text{in}}}+(\hat{c}f^+,q)\label{eq:discrete2}
\end{align}
for any $(\bm{\psi},q)\in W_h\times U_h$. Note that the last two terms on the left-hand side of \eqref{eq:discrete2} are the boundary correction terms, which are used to improve the stability estimate. Here, the bilinear forms are defined by
\begin{align*}
T_h^*(c_h,\bm{\psi})&=\sum_{\tau \in \mathcal{T}_h}(c_h, \nabla\cdot \bm{\psi})_\tau - \sum_{e \in \mathcal{F}_{pr}} (c_h, \jump{\bm{\psi}} \cdot \bm{n})_e,\\
T_h(\bm{z}_h,q)&=-\sum_{\tau \in \mathcal{T}_h}(\bm{z}_h,\nabla q)_\tau+\sum_{e\in \mathcal{F}_{dl}}(\bm{z}_h\cdot\bm{n},\jump{q})_e,\\
S_h(c_h,q)&=\sum_{e \in \mathcal{F}_{dl}}\int_e\avg{c_h}\jump{q } \bm{u}_h\cdot \bm{n}~ds+\frac{1}{2}\sum_{e \in \mathcal{F}_{dl}}\int_e\jump{c_h}\jump{q} |\bm{u}_h\cdot \bm{n}|~ds.
\end{align*}
The initial condition $c_h(\cdot,0)$ is defined as $c_h(\cdot,0)=c_h^0$, where $c_h^0$ is the $L^2$-orthogonal projection of $c^0$.

Integration by parts implies the discrete adjoint property
\begin{align}
T_h(\bm{\psi},q)=T_h^*(q,\bm{\psi})\quad \forall (\bm{\psi},q)\in W_h\times U_h.\label{eq:adjointTh}
\end{align}
Let
\begin{align*}
A_{\bm{u}_h}(\bm{z}_h,c_h;\bm{\psi}_h,q_h)&=(K^{-1} \bm{z}_h, \bm{\psi}_h) - T_h^*(c_h,\bm{\psi}_h)+(\phi \frac{\partial c_h}{\partial t},q_h)+(c_hf^-,q)+T_h(\bm{z}_h,q_h)\\
&\;-(\bm{u}_hc_h,\nabla q_h)+S_h(c_h,q_h)+
(c_h\bm{u}_h\cdot\bm{n},q_h)_{\Gamma_{\text{out}}}+\frac{1}{2}((\bm{u}-\bm{u}_h)\cdot\bm{n}c_h,q_h)_{\Gamma_{\text{out}}}\\
&\;-\frac{1}{2}((\bm{u}-\bm{u}_h)\cdot\bm{n}c_h,q_h)_{\Gamma_{\text{in}}}.
\end{align*}
Then, it follows from \eqref{eq:discrete1}-\eqref{eq:discrete2} that
\begin{align}
A_{\bm{u}_h}(\bm{z}_h,c_h;\bm{\psi}_h,q_h)=(\phi s, q_h)-
(c_{\text{in}}\bm{u}\cdot\bm{n},q_h)_{\Gamma_{\text{in}}}+(\hat{c}f^+,q_h).\label{eq:Ah}
\end{align}
Replacing $(\bm{z}_h,c_h)$ by $(\bm{z},c)$ in \eqref{eq:Ah}, we can infer that the weak solution $(\bm{z},c)$ satisfies
\begin{align*}
A_{\bm{u}}(\bm{z},c;\bm{\psi}_h,q_h)=(\phi s, q_h)-
(c_{\text{in}}\bm{u}\cdot\bm{n},q_h)_{\Gamma_{\text{in}}}+(\hat{c}f^+,q_h),
\end{align*}
where
\begin{align*}
A_{\bm{u}}(\bm{z},c;\bm{\psi}_h,q_h)&=(K^{-1} \bm{z}, \bm{\psi}_h) - T_h^*(c,\bm{\psi}_h)+(\phi \frac{\partial c}{\partial t},q_h)+(cf^-,q_h)+T_h(\bm{z},q_h)-(\bm{u}c,\nabla q_h)\\
&\;+S_h(c,q_h)+
(c\bm{u}\cdot\bm{n},q_h)_{\Gamma_{\text{out}}}.
\end{align*}
Thus
\begin{align}
A_{\bm{u}}(\bm{z},c;\bm{\psi}_h,q_h)-A_{\bm{u}_h}(\bm{z}_h,c_h;\bm{\psi}_h,q_h)=0\quad \forall (\bm{\psi}_h,q_h)\in W_h\times U_h .
\label{eq:errorAh}
\end{align}
Similar to \eqref{eq:inf-sup}, the following inf-sup condition holds
\begin{align}
\|q\|_{1,h,*}\leq C\sup_{\bm{\psi}\in W_h}\frac{T_h(\bm{\psi},q)}{\|\bm{\psi}\|_{0}}\quad \forall q\in U_h.\label{eq:inf-supTh}
\end{align}

%
%

The following lemma is found to be useful for later analysis (cf. \cite{Cockburn99}).
\begin{lemma}\label{lemma:inequalityTime}
Suppose that for all $T>0$
\begin{align*}
\chi^2(T)+R(T)\leq A(T)+2\int_0^T B(t) \chi(t)\;dt,
\end{align*}
where $R, A$ and $B$ are nonnegative functions. Then
\begin{align*}
\sqrt{\chi^2+R(T)}\leq \sup_{0\leq t\leq T}A^{1/2}(t)+\int_0^T B(t)\;dt.
\end{align*}

\end{lemma}

For later analysis, we define
\begin{align*}
\|(c_h,\bm{z}_h)\|_c^2=\|\phi^{1/2}c_h(T)\|_0^2+2\int_0^T \|K^{-1/2}\bm{z}_h\|_0^2\;dt.
\end{align*}

\begin{theorem}\label{thm:stability}
(stability).
Let $(\bm{z}_h,c_h)\in W_h\times U_h$ be the discrete solution of \eqref{eq:discrete1}-\eqref{eq:discrete2}. Then, the following stability result holds
\begin{equation}
\begin{split}
\|(c_h,\bm{z}_h)\|_c&\leq \Big( \phi^*\|c^0\|_0^2+\int_0^T ( (|\bm{u}\cdot\bm{n}|,c_{\textnormal{in}}^2)_{\Gamma_{\textnormal{in}}}+(\hat{c}^2,f^+)_{\Omega_D})\;dt\Big)^{\frac{1}{2}}+\int_0^T \|\phi^{\frac{1}{2}}s\|_0\;dt.
\end{split}
\end{equation}

\end{theorem}

\begin{proof}

From the definition of $A_{\bm{u}_h}(\cdot,\cdot;\cdot,\cdot)$ and \eqref{eq:adjointTh}, we have
\begin{align*}
A_{\bm{u}_h}(\bm{z}_h,c_h;\bm{z}_h,c_h)&=(K^{-1} \bm{z}_h, \bm{z}_h) - T_h^*(c_h,\bm{z}_h)+(\phi \frac{\partial c_h}{\partial t},c_h)+(c_hf^-,c_h)+T_h(\bm{z}_h,c_h)\\
&\;-(\bm{u}_hc_h,\nabla c_h)+S_h(c_h,c_h)+
(\bm{u}_h\cdot\bm{n}c_h,c_h)_{\Gamma_{\text{out}}}\\
&\;+\frac{1}{2}((\bm{u}-\bm{u}_h)\cdot\bm{n}c_h,c_h)_{\Gamma_{\text{out}}}
-\frac{1}{2}((\bm{u}-\bm{u}_h)\cdot\bm{n}c_h,c_h)_{\Gamma_{\text{in}}}\\
&=(K^{-1} \bm{z}_h, \bm{z}_h) +(\phi \frac{\partial c_h}{\partial t},c_h)-(\bm{u}_hc_h,\nabla c_h)+(c_hf^-,c_h)+S_h(c_h,c_h)\\
&\;+
(\bm{u}_h\cdot\bm{n}c_h,c_h)_{\Gamma_{\text{out}}}+\frac{1}{2}((\bm{u}-\bm{u}_h)\cdot\bm{n}c_h,c_h)_{\Gamma_{\text{out}}}
-\frac{1}{2}((\bm{u}-\bm{u}_h)\cdot\bm{n}c_h,c_h)_{\Gamma_{\text{in}}}.
\end{align*}
The third term on the right-hand side can be recast into the following form via integration by parts
\begin{align*}
(\bm{u}_hc_h,\nabla c_h)&=\frac{1}{2}\sum_{\tau\in \mathcal{T}_{h}}(\bm{u}_h\cdot\bm{n},c_h^2)_{\partial \tau}-\frac{1}{2}\sum_{\tau\in \mathcal{T}_{h}}(c_h^2,\nabla \cdot\bm{u}_h)_\tau\\
&=\sum_{e\in \mathcal{F}_{dl}}(\bm{u}_h\cdot\bm{n},\jump{c_h}\avg{c_h})_{e}
+\frac{1}{2}(\bm{u}_h\cdot\bm{n},c_hc_h)_{\partial \Omega}-\frac{1}{2}(c_h^2,\nabla \cdot\bm{u}_h),
\end{align*}
where we use the fact that $\bm{u}_h\in H(\tdiv;\Omega)$.

Thus, we have
\begin{align*}
A_{\bm{u}_h}(\bm{z}_h,c_h;\bm{z}_h,c_h)&=\|K^{-\frac{1}{2}} \bm{z}_h\|_0^2 +(\phi \frac{\partial c_h}{\partial t},c_h)+\frac{1}{2}(c_h^2,\nabla \cdot\bm{u}_h)+(c_hf^-,c_h)\\
&\;+\frac{1}{2}\sum_{e \in \mathcal{F}_{dl}}\int_e\jump{c_h}^2 |\bm{u}_h\cdot \bm{n}_e|~ds+\frac{1}{2}(\bm{u}\cdot\bm{n}c_h,c_h)_{\Gamma_{\text{out}}}-\frac{1}{2}(\bm{u}\cdot\bm{n}c_h,c_h)_{\Gamma_{\text{in}}}.
\end{align*}
Integrating over $T$ yields
\begin{equation}
\begin{split}
&\int_0^T A_{\bm{u}_h}(\bm{z}_h,c_h;\bm{z}_h,c_h)\;dt=\int_0^T\Big(\|K^{-\frac{1}{2}} \bm{z}_h\|_0^2 +(\phi \frac{\partial c_h}{\partial t},c_h)+\frac{1}{2}(c_h^2,\nabla \cdot\bm{u}_h)+(c_hf^-,c_h)\\
&\;+\frac{1}{2}\sum_{e \in \mathcal{F}_{dl}}\int_e\jump{c_h}^2 |\bm{u}_h\cdot \bm{n}|~ds+\frac{1}{2}(\bm{u}\cdot\bm{n}c_h,c_h)_{\Gamma_{\text{out}}}-\frac{1}{2}(\bm{u}\cdot\bm{n}c_h,c_h)_{\Gamma_{\text{in}}}\Big)\;dt\\
&=\frac{1}{2}(\|\phi^{\frac{1}{2}}c_h(T)\|_0^2-\|\phi^{\frac{1}{2}}c_h(0)\|_0^2)
+\int_0^T\|K^{-\frac{1}{2}} \bm{z}_h\|_0^2+\frac{1}{2}\int_0^T(c_h^2,\nabla \cdot\bm{u}_h)+\int_0^T(c_hf^-,c_h)\;dt \\
&\;+\frac{1}{2}\int_0^T\sum_{e \in \mathcal{F}_{dl}}\int_e\jump{c_h}^2
 |\bm{u}_h\cdot \bm{n}|~ds+\frac{1}{2}\int_0^T(|\bm{u}\cdot\bm{n}|c_h,c_h)_{\partial \Omega}\;dt.
\end{split}
\label{eq:Auh}
\end{equation}
Therefore, it follows from \eqref{eq:Ah} that
\begin{equation}
\begin{split}
&\frac{1}{2}(\|\phi^{\frac{1}{2}}c_h(T)\|_0^2-\|\phi^{\frac{1}{2}}c_h(0)\|_0^2)+\int_0^T\|K^{-\frac{1}{2}} \bm{z}_h\|_0^2\;dt+\frac{1}{2}\int_0^T(c_h^2,\nabla \cdot\bm{u}_h)\;dt+\int_0^T(c_h^2, f^-)\;dt\\
&\;+\frac{1}{2}\int_0^T\sum_{e \in \mathcal{F}_{dl}}\int_e\jump{c_h}^2 |\bm{u}_h\cdot \bm{n}|~ds\;dt+\frac{1}{2}\int_0^T(|\bm{u}\cdot\bm{n}|c_h,c_h)_{\partial \Omega}\;dt\\
&=\int_0^T ((\phi s, c_h)+(\hat{c}f^+,c_h)-
(c_{\text{in}}\bm{u}\cdot\bm{n},c_h)_{\Gamma_{\text{in}}})\;dt.
\end{split}
\label{eq:ChT}
\end{equation}
Note that $\nabla\cdot \bm{u}_h\mid_{\Omega_B}=0$ and $\nabla\cdot \bm{u}_h\mid_{\Omega_D}=f$, we can deduce that
\begin{align*}
\frac{1}{2}(c_h^2,\nabla \cdot\bm{u}_h)+(c_hf^-,c_h)=\left\{
                              \begin{array}{ll}
                                \frac{1}{2}(c_h^2,f)_{\Omega_D}, & \hbox{$\nabla \cdot\bm{u}_h\geq 0$}, \\
                                -\frac{1}{2}(c_h^2,f)_{\Omega_D}, & \hbox{$\nabla \cdot\bm{u}_h<0$}.
                              \end{array}
                            \right.
\end{align*}
%
Furthermore, an application of the Cauchy-Schwarz inequality implies
\begin{align*}
(\hat{c}f^+,c_h)\leq (\hat{c}^2,f^+)_{\Omega_D}^{\frac{1}{2}}(c_h^2,f^+)_{\Omega_D}^{\frac{1}{2}}\leq
\frac{1}{2}\Big((\hat{c}^2,f^+)_{\Omega_D}+(c_h^2,f^+)_{\Omega_D}\Big).
\end{align*}
Therefore, we can infer that
\begin{equation}
\begin{split}
\frac{1}{2}\|\phi^{\frac{1}{2}}c_h(T)\|_0^2+\int_0^T\|K^{-\frac{1}{2}} \bm{z}_h\|_0^2&\leq \frac{1}{2}\|\phi^{\frac{1}{2}}c_h(0)\|_0^2+\frac{1}{2}\int_0^T \Big((|\bm{u}\cdot\bm{n}|,c_{\text{in}}^2)_{\Gamma_{\text{in}}}+(\hat{c}^2,f^+)_{\Omega_D}\Big)\;dt\\
&\;+\int_0^T\|\phi^{\frac{1}{2}}s\|_0\|\phi^{\frac{1}{2}}c_h\|_0\;dt,
\end{split}
\label{eq:ch}
\end{equation}
where we use Young's inequality for the last term on the right-hand side of \eqref{eq:ChT}.


Recall that $c_h(0)$ is $L^2$-orthogonal projection of $c^0$, thereby it holds
\begin{align*}
\|\phi^{\frac{1}{2}}c_h(0)\|_0\leq (\phi^*)^{\frac{1}{2}}\|c^0\|_0,
\end{align*}
Then an application of Lemma~\ref{lemma:inequalityTime} completes the proof.

\end{proof}

\begin{remark}
We can observe from Theorem~\ref{thm:stability} that our stability estimate is sharp in the sense that there is no undetermined constant in front of the right-hand side, which benefits from the strong mass conservation of the proposed scheme. On the other hand, the introduction of the boundary correction terms $\frac{1}{2}((\bm{u}-\bm{u}_h)\cdot\bm{n}c_h,q)_{\Gamma_{\textnormal{out}}}$ and $
-\frac{1}{2}((\bm{u}-\bm{u}_h)\cdot\bm{n}c_h,q)_{\Gamma_{\textnormal{in}}}$ improves the stability estimate for the transport equation. Indeed, the stability estimate depends on the exact velocity on the inflow boundary rather than on the approximated velocity.

\end{remark}

%

%
%

In the proof of the next lemma, we still use $I_h$ and $J_h$ to represent the interpolation error estimates which follows the same definitions given in \eqref{eq:Ih} and \eqref{eq:Jh} but extended to the global domain $\Omega$.

\begin{theorem}\label{thm:convergence-transport}
Let $(\bm{z}_h,c_h)\in W_h\times U_h$ be the discrete solution of \eqref{eq:discrete1}-\eqref{eq:discrete2}. Then, the following convergence error estimate holds
\begin{align*}
\|(c-c_h,\bm{z}-\bm{z}_h)\|_c &\leq C \Big(h^{k+1}\|c\|_{C(0,T;H^{k+1}(\Omega))}+h^{k+1}\Big(\int_0^T (\|c\|_{k+1}+\|c_t\|_{k+1}\\
&\;+\|\bm{z}\|_{k+1}+\|\bm{u}\|_{k+1}+\|p\|_{k+1}+\epsilon^{-\frac{1}{2}}\|\bm{L}\|_{k+1,\Omega_B}+\|\bm{u}\|_{k+1,\partial \Omega})\;dt\Big)\Big),
\end{align*}
where
\begin{align*}
\|c\|_{C(0,T;H^m(\Omega))}=\max_{0\leq t\leq T}\|c(t)\|_{H^m(\Omega)}.
\end{align*}

\end{theorem}

\begin{proof}

We can infer from \eqref{eq:errorAh} that
\begin{equation}
\begin{split}
&A_{\bm{u}_h}(\bm{z}_h-J_h\bm{z},c_h-I_hc;\bm{z}_h-J_h\bm{z},c_h-I_hc)
=A_{\bm{u}}(\bm{z}-J_h\bm{z},c-I_hc;\bm{z}_h-J_h\bm{z},c_h-I_hc)\\
&\;+A_{\bm{u}}(J_h\bm{z},I_hc;\bm{z}_h-J_h\bm{z},c_h-I_hc)
-A_{\bm{u}_h}(J_h\bm{z},I_hc;\bm{z}_h-J_h\bm{z},c_h-I_hc).
\end{split}
\label{eq:error1}
\end{equation}
Proceeding similarly to \eqref{eq:Auh}, we have
\begin{equation}
\begin{split}
&\int_0^T A_{\bm{u}_h}(\bm{z}_h-J_h\bm{z},c_h-I_hc;\bm{z}_h-J_h\bm{z},c_h-I_hc)=\frac{1}{2}(\|\phi^{\frac{1}{2}}(c_h-I_hc)(T)\|_0^2\\
&\;-\|\phi^{\frac{1}{2}}(c_h-I_hc)(0)\|_0^2)
+\int_0^T\|K^{-\frac{1}{2}} (\bm{z}_h-J_h\bm{z})\|_0^2+\frac{1}{2}\int_0^T((c_h-I_hc)^2,\nabla \cdot\bm{u}_h)\\
&\;+\int_0^T ((c_h-I_hc)^2,f^-)
+\frac{1}{2}\int_0^T\sum_{e \in \mathcal{F}_{dl}}\int_e\jump{c_h-I_hc}^2
 |\bm{u}_h\cdot \bm{n}|~ds\\
 &\;+\frac{1}{2}\int_0^T(|\bm{u}\cdot\bm{n}|(c_h-I_hc),c_h-I_hc)_{\partial \Omega}\;dt.
 \end{split}
 \label{eq:A}
\end{equation}
The first term on the right-hand side of \eqref{eq:error1} can be rewritten as follows by using the definitions of $I_h$ and $J_h$
\begin{align*}
&A_{\bm{u}}(\bm{z}-J_h\bm{z},c-I_hc;\bm{z}_h-J_h\bm{z},c_h-I_hc)\\
&=(K^{-1} (\bm{z}-J_h\bm{z}), \bm{z}_h-J_h\bm{z})+(\phi \frac{\partial (c-I_hc)}{\partial t},c_h-I_hc)-(\bm{u}(c-I_hc),\nabla (c_h-I_hc))\\
&\;+S_h(c-I_hc,c_h-I_hc)+
(\bm{u}\cdot\bm{n}(c-I_hc),c_h-I_hc)_{\Gamma_{\text{out}}}:=\sum_{i=1}^5 L_i.
\end{align*}
Now we estimate $L_i, i=1,\ldots,5$. First, the Cauchy-Schwarz inequality yields
\begin{align*}
L_1&\leq  \|K^{-1/2} (\bm{z}-J_h\bm{z})\|_0\|K^{-1/2}( \bm{z}_h-J_h\bm{z})\|_0,\\
L_2&= (\phi \frac{\partial (c-I_hc)}{\partial t},c_h-I_hc)\leq (\phi^*)^{1/2} \|\frac{\partial (c-I_hc)}{\partial t}\|_0\|\phi^{1/2}(c_h-I_hc)\|_0.
\end{align*}
The inf-sup condition \eqref{eq:inf-supTh} and the error equation \eqref{eq:errorAh} imply that
\begin{align}
\|c_h-I_hc\|_{1,h,*}\leq  CK_{\text{min}}^{-1/2}\|K^{-1/2}(\bm{z}-\bm{z}_h)\|_0.\label{eq:Ihz}
\end{align}
Then an application of the Cauchy-Schwarz inequality implies
\begin{align*}
\sum_{i=3}^4L_i:&=-(\bm{u}(c-I_hc),\nabla (c_h-I_hc))+\sum_{e \in \mathcal{F}_{dl}}\int_e\avg{c-I_hc}\jump{c_h-I_hc} \bm{u}\cdot \bm{n}~ds\\
&\;+\frac{1}{2}\sum_{e \in \mathcal{F}_{dl}}\int_e\jump{c-I_hc}\jump{c_h-I_hc }|\bm{u}\cdot \bm{n}|~ds\\
&\leq C (\|\bm{u}\|_{L^\infty(\Omega)}\|c-I_hc\|_0+\sum_{e\in \mathcal{F}_{dl}}\|\bm{u}\|_{L^\infty(e)}h^{1/2}\|c-I_hc\|_{0,e})
\|c_h-I_hc\|_{1,h,*}\\
&\leq C (\|\bm{u}\|_{L^\infty(\Omega)}\|c-I_hc\|_0+\sum_{e\in \mathcal{F}_{dl}}\|\bm{u}\|_{L^\infty(e)}h^{1/2}\|c-I_hc\|_{0,e})
\|K^{-1/2}(\bm{z}-\bm{z}_h)\|_0.
\end{align*}
The trace inequality \eqref{eq:Gammaz} implies that
\begin{align*}
\|c_h-I_hc\|_{0,\Gamma_{\text{out}}}\leq C \|c_h-I_hc\|_{1,h,*}\leq C K_{\text{min}}^{-\frac{1}{2}}\|K^{-\frac{1}{2}}(\bm{z}-\bm{z}_h)\|_0.
\end{align*}
Let $\bar{\bm{u}}$ denote the average value of $\bm{u}$ over $e\in \Gamma_{\text{out}}$, then we can infer from the definition of $I_h$, \eqref{eq:Gammaz} and \eqref{eq:Ihz} that
\begin{align*}
L_5=((c-I_hc)(\bm{u}-\bar{\bm{u}})\cdot\bm{n},c_h-I_hc)_{\Gamma_{\text{out}}}\leq C (\sum_{e\in \Gamma_{\text{out}}}h_e^{1/2} \|c-I_hc\|_{0,e}\|\bm{u}\|_{1,\infty,\tau_e})\|K^{-1/2}(\bm{z}-\bm{z}_h)\|_0,
\end{align*}
where we use $\tau_e$ to represent the element having the edge $e$, i.e., $e\subset \partial \tau_e$.

It remains to show the upper bound for the last two terms on the right-hand side of \eqref{eq:error1}. First, we have
\begin{align*}
&A_{\bm{u}}(J_h\bm{z},I_hc;\bm{z}_h-J_h\bm{z},c_h-I_hc)
-A_{\bm{u}_h}(J_h\bm{z},I_hc;\bm{z}_h-J_h\bm{z},c_h-I_hc)\\
&=-((\bm{u}-\bm{u}_h) I_hc,\nabla (c_h-I_hc))+\sum_{e \in \mathcal{F}_{dl}}\int_e\avg{I_hc}\jump{c_h-I_hc}(\bm{u}- \bm{u}_h)\cdot \bm{n}~ds\\
&\;+\frac{1}{2}\sum_{e \in \mathcal{F}_{dl}}\int_e\jump{I_hc}\jump{c_h-I_hc} |(\bm{u}-\bm{u}_h)\cdot \bm{n}|~ds+
((\bm{u}-\bm{u}_h)\cdot\bm{n} I_hc,c_h-I_hc)_{\Gamma_{\text{out}}}\\
&\;-\frac{1}{2}((\bm{u}-\bm{u}_h)\cdot\bm{n}I_hc,c_h-I_hc)_{\Gamma_{\text{out}}}+
\frac{1}{2}((\bm{u}-\bm{u}_h)\cdot\bm{n}I_hc,c_h-I_hc)_{\Gamma_{\text{in}}}.
\end{align*}
The first term on the right-hand side can be written as follows by using integration by parts
\begin{align*}
&-((\bm{u}-\bm{u}_h) I_hc,\nabla (c_h-I_hc))\\
&=(\nabla\cdot((\bm{u}-\bm{u}_h) I_hc), c_h-I_hc)-\sum_{e\in \mathcal{F}_{dl}}((\bm{u}-\bm{u}_h)\cdot\bm{n}\jump{I_hc}, \avg{c_h-I_hc})_e\\
&\;-\sum_{e\in \mathcal{F}_{dl}}((\bm{u}-\bm{u}_h)\cdot\bm{n}\avg{I_hc}, \jump{c_h-I_hc})_e-((\bm{u}-\bm{u}_h)\cdot\bm{n}I_hc, c_h-I_hc)_{\partial \Omega}.
\end{align*}
Thus, we have
\begin{align*}
&A_{\bm{u}}(J_h\bm{z},I_hc;\bm{z}_h-J_h\bm{z},c_h-I_hc)
-A_{\bm{u}_h}(J_h\bm{z},I_hc;\bm{z}_h-J_h\bm{z},c_h-I_hc)\\
&=(\nabla\cdot((\bm{u}-\bm{u}_h) I_hc), c_h-I_hc)-\frac{1}{2}((\bm{u}-\bm{u}_h)\cdot\bm{n}I_hc, c_h-I_hc)_{\partial \Omega}\\
&\;+\frac{1}{2}\sum_{e \in \mathcal{F}_{dl}}\int_e\jump{I_hc}\jump{c_h-I_hc} |(\bm{u}-\bm{u}_h)\cdot \bm{n}|~ds\\
&\;-\sum_{e\in \mathcal{F}_{dl}}((\bm{u}-\bm{u}_h)\cdot\bm{n}\jump{I_hc}, \avg{c_h-I_hc})_e:=\sum_{i=1}^4 R_i.
\end{align*}
$R_1$ can be estimated by using integration by parts and the Cauchy-Schwarz inequality
\begin{align*}
R_1&=(\nabla (I_hc) (\bm{u}-\bm{u}_h),c_h-I_hc)+(I_hc \nabla \cdot (\bm{u}-\bm{u}_h),c_h-I_hc)\\
&\leq \|\nabla (I_hc)\|_{L^\infty(\Omega)}\|\bm{u}-\bm{u}_h\|_0\|\phi^{1/2}(c_h-I_hc)\|_0,
\end{align*}
where we use the fact that $\nabla\cdot(\bm{u}-\bm{u}_h)=0$.

The Cauchy-Schwarz inequality, \eqref{eq:Gammaz} and \eqref{eq:Ihz} imply
\begin{align*}
|R_2|=\frac{1}{2}|((\bm{u}-\bm{u}_h)\cdot\bm{n}I_hc, c_h-I_hc)_{\partial \Omega}|&\leq \frac{1}{2}\|I_hc\|_{L^\infty(\partial \Omega)}\|(\bm{u}-\bm{u}_h)\cdot\bm{n}\|_{0,\partial \Omega}
\|c_h-I_hc\|_{0,\partial \Omega}\\
&\leq C\|I_hc\|_{L^\infty(\partial \Omega)}\|(\bm{u}-\bm{u}_h)\cdot\bm{n}\|_{0,\partial \Omega}
\|c_h-I_hc\|_{1,h,*}\\
&\leq C \|I_hc\|_{L^\infty(\partial \Omega)}\|(\bm{u}-\bm{u}_h)\cdot\bm{n}\|_{0,\partial \Omega}
\|K^{-1}(\bm{z}-\bm{z}_h)\|_0.
\end{align*}
We can deduce from the Cauchy-Schwarz inequality that
\begin{align*}
&R_3+R_4=\frac{1}{2}\sum_{e \in \mathcal{F}_{dl}}\int_e\jump{I_hc}\jump{c_h-I_hc} |(\bm{u}-\bm{u}_h)\cdot \bm{n}|~ds-\sum_{e\in \mathcal{F}_{dl}}((\bm{u}-\bm{u}_h)\cdot\bm{n}\jump{I_hc}, \avg{c_h-I_hc})_e\\
&\leq C\sum_{e \in \mathcal{F}_{dl}}\|c-I_hc\|_{L^\infty(e)}\|(\bm{u}-\bm{u}_h)\cdot\bm{n}\|_{0,e}\|c_h-I_hc\|_{0,e}.
\end{align*}
%
Thus, the inverse inequality leads to
\begin{align*}
R_3+R_4\leq C\sum_{e\in \mathcal{F}_{dl}}\|c\|_{W^{1,\infty}(e)} h^{1/2} \|(\bm{u}-\bm{u}_h)\cdot\bm{n}\|_{0,e}\|\phi^{1/2}(c_h-I_hc)\|_0.
\end{align*}
Then, integrating over $t$ for both sides of \eqref{eq:error1} and using \eqref{eq:A} imply that
\begin{align*}
\frac{1}{2}\|\phi^{\frac{1}{2}}(I_hc-c_h)(T)\|_0^2+\int_0^T\|K^{-\frac{1}{2}} (J_h\bm{z}-\bm{z}_h)\|_0^2&\leq \frac{1}{2}\|\phi^{\frac{1}{2}}(I_hc-c_h)(0)\|_0^2+\int_0^T (\sum_{i=1}^5 L_i+\sum_{i=1}^4 R_i).
\end{align*}
%
%
%
%
%
%
%
%
An appeal to Lemma~\ref{lemma:inequalityTime} implies that
\begin{align*}
\|(I_hc-c_h,J_h\bm{z}-\bm{z}_h)\|_c&\leq C\Big(\|\phi^{\frac{1}{2}}(I_hc-c_h)(0)\|_0^2+\int_0^T\Big((\sum_{e\in \mathcal{F}_h}h_e\|(\bm{u}-\bm{u}_h)\cdot\bm{n}\|_{0,e}^2)^{\frac{1}{2}}\\
&\;+\|\bm{u}-\bm{u}_h\|_0
+\|K^{-\frac{1}{2}}(\bm{z}-J_h\bm{z})\|_0+\|K^{-\frac{1}{2}}(\bm{z}-\bm{z}_h)\|_0
+\|\frac{\partial (c-I_hc)}{\partial t}\|_0\\
&\;+\|c-I_hc\|_0+(\sum_{e\in \mathcal{F}_h}h_e\|c-I_hc\|_{0,e}^2)^{\frac{1}{2}}+\|(\bm{u}-\bm{u}_h)\cdot\bm{n}\|_{0,\partial \Omega}\Big)\;dt\Big).
\end{align*}
Then an appeal to the triangle inequality, \eqref{eq:BDM}, \eqref{eq:Perror}, \eqref{eq:uerror}, Theorem~\ref{thm:L2} and \eqref{eq:interpGin} completes the proof.

\end{proof}

Now we briefly introduce the fully discrete scheme for \eqref{eq:discrete1}-\eqref{eq:discrete2} based on the backward Euler scheme. We introduce a partition of the time interval $[0,T]$ into
 subintervals $[t_{n},t_{n+1}],0\leq n\leq N (N\;\mbox{is an integer})$ and denote the time step size by $\Delta t=\frac{T}{N}$. Using the backward Euler scheme in time, we get the fully discrete scheme as follows: Find $(\bm{z}_h^{n+1},c_h^{n+1})\in W_h\times U_h$ such that
\begin{align}
&(K^{-1} \bm{z}_h^{n+1}, \bm{\psi}) - T_h^*(c_h^{n+1},\bm{\psi})=0,\label{eq:discrete1-fully}\\
&(\phi \frac{c_h^{n+1}-c_h^n}{\Delta t},q)+(c_h^{n+1}f^-,q)+T_h(\bm{z}_h^{n+1},q)-(\bm{u}_hc_h^{n+1},\nabla q)+S_h(c_h^{n+1},q)+(\bm{u}_hc_h^{n+1}\cdot\bm{n},q)_{\Gamma_{\text{out}}}\nonumber\\
&\;+\frac{1}{2}((\bm{u}-\bm{u}_h)\cdot\bm{n}c_h^{n+1},q)_{\Gamma_{\text{out}}}
-\frac{1}{2}((\bm{u}-\bm{u}_h)\cdot\bm{n}c_h^{n+1},q)_{\Gamma_{\text{in}}} = (\phi s^{n+1}, q)-
(c_{\text{in}}\bm{u}\cdot\bm{n},q)_{\Gamma_{\text{in}}}+(\hat{c}f^+,q)\label{eq:discrete2-fully}
\end{align}
for any $(\bm{\psi},q)\in W_h\times U_h$.

We can show the following stability result for the fully discrete scheme.
\begin{theorem}
Let $(\bm{z}_h^{n},c_h^n)\in W_h\times U_h$ be the discrete solution of \eqref{eq:discrete1-fully}-\eqref{eq:discrete2-fully}. Then the following stability result holds
\begin{align*}
2\Delta t\sum_{n=0}^N\|K^{-\frac{1}{2}} \bm{z}_h^{n+1}\|_0^2+\|\phi^{\frac{1}{2}}c_h^{N+1}\|_0^2
&\leq C\Big(\sum_{n=0}^N\Delta t\|\phi^{\frac{1}{2}} s^{n+1}\|_0^2+\sum_{n=0}^N\Delta t (\bm{u}\cdot\bm{n},c_{\textnormal{in}}^2)_{\Gamma_{\textnormal{in}}}\\
&\;+\sum_{n=0}^N\Delta t(\hat{c}^2,f^+)_{\Omega_D}+\phi^*\|c^{0}\|_0^2\Big).
\end{align*}

\end{theorem}

\begin{proof}
First, we have
\begin{align*}
A_{\bm{u}_h}(\bm{z}_h^{n+1},c_h^{n+1};\bm{z}_h^{n+1},c_h^{n+1})
&=\|K^{-1/2} \bm{z}_h^{n+1}\|_0^2+(\phi \frac{c_h^{n+1}-c_h^{n}}{\Delta t},c_h^{n+1})+\frac{1}{2}((c_h^{n+1})^2,\nabla\cdot\bm{u}_h)\\
&\;+\frac{1}{2}\sum_{e\in \mathcal{F}_{dl}}(\jump{c_h^{n+1}}^2,|\bm{u}_h\cdot\bm{n}|)_e+(c_h^{n+1}f^-,c_h^{n+1})
+(\bm{u}_h\cdot\bm{n}c_h^{n+1},c_h^{n+1})_{\Gamma_{\text{out}}}\\
&-\frac{1}{2}(\bm{u}_h\cdot\bm{n},(c_h^{n+1})^2)_{\partial \Omega}+
\frac{1}{2}((\bm{u}-\bm{u}_h)\cdot\bm{n}c_h^{n+1},c_h^{n+1})_{\Gamma_{\text{out}}}\\
&\;-\frac{1}{2}((\bm{u}-\bm{u}_h)\cdot\bm{n}c_h^{n+1},c_h^{n+1})_{\Gamma_{\text{in}}}.
\end{align*}
where we exploit integration by parts for $(\bm{u}_hc_h^{n+1},\nabla c_h^{n+1})$.

Noting that $(a-b)a=\frac{a^2-b^2+(a-b)^2}{2}$, we can deduce that
\begin{align*}
&\|K^{-1/2} \bm{z}_h^{n+1}\|_0^2+\frac{1}{2\Delta t}(\|\phi^{\frac{1}{2}}c_h^{n+1}\|_0^2-\|\phi^{\frac{1}{2}}c_h^{n}\|_0^2+\|\phi^{\frac{1}{2}}(c_h^{n+1}-c_h^n)\|_0^2)
+\frac{1}{2}((c_h^{n+1})^2,\nabla\cdot\bm{u}_h)\\
&\;+\frac{1}{2}\sum_{e\in \mathcal{F}_{dl}}(\jump{c_h^{n+1}}^2,|\bm{u}_h\cdot\bm{n}|)_e+(c_h^{n+1}f^-,c_h^{n+1})
+\frac{1}{2}(|\bm{u}\cdot\bm{n}|,(c_h^{n+1})^2)_{\partial \Omega}\\
&=(\phi s^{n+1}, c_h^{n+1})-
(c_{\text{in}}\bm{u}\cdot\bm{n},c_h^{n+1})_{\Gamma_{\text{in}}}+(\hat{c}f^+,c_h^{n+1}).
\end{align*}
Then proceeding similarly to \eqref{eq:ch}, we can infer that
\begin{align*}
\|K^{-\frac{1}{2}} \bm{z}_h^{n+1}\|_0^2+\frac{1}{2\Delta t}(\|\phi^{\frac{1}{2}}c_h^{n+1}\|_0^2-\|\phi^{\frac{1}{2}}c_h^{n}\|_0^2)&\leq \frac{1}{2}\|\phi^{\frac{1}{2}} s^{n+1}\|_0^2+\frac{1}{2}(\bm{u}\cdot\bm{n},c_{\text{in}}^2)_{\Gamma_{\text{in}}}\\
&\;+\frac{1}{2}(\hat{c}^2,f^+)_{\Omega_D}+\frac{1}{2}\|\phi^{\frac{1}{2}}c_h^{n+1}\|_0^2,
\end{align*}

%
Making a summation for $n$ from $0$ to $N$ and using the discrete Gronwall inequality imply that
\begin{align*}
2\Delta t\sum_{n=0}^N\|K^{-\frac{1}{2}} \bm{z}_h^{n+1}\|_0^2+\|\phi^{\frac{1}{2}}c_h^{N+1}\|_0^2
&\leq C\Big(\sum_{n=0}^N\Delta t\|\phi^{\frac{1}{2}} s^{n+1}\|_0^2+\sum_{n=0}^N\Delta t(\bm{u}\cdot\bm{n},c_{\text{in}}^2)_{\Gamma_{\text{in}}}\\
&\;+\sum_{n=0}^N\Delta t(\hat{c}^2,f^+)_{\Omega_D}+\phi^*\|c^{0}\|_0^2\Big).
\end{align*}
Therefore, the proof is completed.

\end{proof}

\begin{remark}
In the above theorem, we have showed the unconditional stability for the fully discrete scheme.
Proceeding similarly to Theorem~\ref{thm:convergence-transport}, and using the discrete Gronwall inequality and Taylor's expansion for the time discretization, we are able to prove the optimal convergence rates for the fully discrete scheme, which are omitted here for simplicity.

\end{remark}

\section{Numerical experiments}\label{sec:numerical}

In this section we present several numerical experiments to verify the proposed theories. We exploit the backward Euler scheme for the time discretization. In the first two tests given below, we take the final simulation time $T=0.1$ and the time step size is $\Delta t=10^{-3}$. In addition, we set $\phi=1$ for all the tests and the interface conditions \eqref{eq:interface1}-\eqref{eq:interface2} are satisfied exactly for all the tests. The polynomial order for all the tests are chosen as $k=1$. In the first three tests, the grids used for the Brinkman discretization in $\Omega_B$ are obtained by first partitioning the domain into rectangles and then dividing each rectangle into the union of triangles by connecting the center point to all the vertices. We also use the same grid for the Darcy region $\Omega_D$. The transport grid in $\Omega$ is the grid used for the flow discretization.

\subsection{Example 1}\label{ex1}
In the first example we set $\Omega_B=(0,1/2)\times (0,1)$ and $\Omega_D=(1/2,1)\times (0,1)$. The velocity field is continuous across the interface and the exact solution is defined by
\begin{align*}
 \bm{u}_D=\begin{cases}
(y(y - 1)(12x^2 - 8x + 1))/4\\
(x(2 x - 1)^2(2y - 1))/4
\end{cases},\quad p_D=x(1/2-x)^2y(1-y)
\end{align*}
and
\begin{align*}
\bm{u}_B=\begin{cases}
x^2 (1/2-x)^2y^2(1-y)^2\\
x^2(1/2-x)^2y^2(1-y)^2
\end{cases}, \quad p_B= x(1/2-x)^2(y-1/2).
\end{align*}
The true solution for the transport equation is defined by $c=t(\cos(\pi x)+\cos(\pi y))/\pi$. The inflow concentration $c_{\text{in}}$ is achieved by evaluating the true concentration at $\Gamma_{\text{in}}$.
We let $K_D=\alpha=1$.
The convergence history for $L^2$ errors of all the variables against the meshsize is displayed in Table~\ref{ex1:table1}-Table~\ref{ex1:table4}, where various values of $K$ and $\epsilon$ are exploited. We can observe that optimal convergence rates can be achieved for different values of $\epsilon$.
As it is well known, the Brinkman equations can describe the Stokes equations and Darcy equations depending on the values of $\epsilon$;
 our numerical results indicate that the proposed scheme can behave uniformly robust for both the Stokes and Darcy limits. In addition, if we take $K$ to be a small number, the convergence rate for $L^2$-error of $\bm{z}$ will degenerate to first order. This can be explained as follows: we use the superconvergence for $\|I_hc-c_h\|_{1,h,*}$ in \eqref{eq:Ihz} that depends on $K_{\text{min}}^{-1/2}$, the control of diffusive flux will be lost due to the dependence of the regularity constant on $K$.

\begin{table}[t]
\begin{center}
{\footnotesize
\begin{tabular}{cc||c c|c c|c c}
\hline
 &Mesh & \multicolumn{2}{|c|}{$\|\epsilon^{-1/2}(\bm{L}-\bm{L}_{h})\|_{0,\Omega_B}$} & \multicolumn{2}{|c|}{$\|\bm{u}_B-\bm{u}_{B,h}\|_{0,\Omega_B}$} & \multicolumn{2}{|c}{$\|p_B-p_{B,h}\|_{0,\Omega_B}$}\\
\hline
$k$ & $h^{-1}$  & Error & Order & Error& Order& Error & Order \\
\hline
1& 2  & 6.238e-04 &   N/A  &7.21e-05 &  N/A &3.97e-04 &   N/A \\
&  4  & 1.364e-04 &  2.19  &1.27e-05 & 2.50 &9.31e-05 & 2.09  \\
&  8  & 3.33e-05 &  2.03  &2.90e-06 & 2.12 &2.32e-05& 2.00 \\
&  16 & 8.30e-06 &  2.00   &7.10e-07 & 2.03 &5.80e-06 &1.99 \\
&  32 & 2.10e-06 & 2.00  &1.75e-07 & 2.01 &1.45e-06 & 1.99  \\
\hline
\end{tabular}}
\caption{Convergence history for Example~\ref{ex1} with $K=1,\epsilon=1$.}
\label{ex1:table1}
\end{center}
\end{table}

\begin{table}[t]
\begin{center}
{\footnotesize
\begin{tabular}{cc||c c|c c|c c|c c}
\hline
 &Mesh & \multicolumn{2}{|c|}{$\|\bm{u}_D-\bm{u}_{D,h}\|_{0,\Omega_D}$} & \multicolumn{2}{|c|}{$\|p_D-p_{D,h}\|_{0,\Omega_D}$} & \multicolumn{2}{|c}{$\|c-c_h\|_{0}$} & \multicolumn{2}{|c}{$\|\bm{z}-\bm{z}_h\|_{0}$}\\
\hline
$k$ & $h^{-1}$  & Error & Order & Error& Order& Error & Order & Error & order \\
\hline
1& 2  & 1.10e-02 &   N/A  &4.60e-03 &  N/A &9.77e-04 &   N/A&3.70e-03 & NA \\
&  4  & 2.90e-03 &  1.94  &2.40e-03 & 0.95 &2.48e-04 & 1.98 &9.36e-04 & 1.98 \\
&  8  & 7.24e-04 &  1.98  &1.20e-03 & 0.99 &6.23e-05 & 1.99 &2.35e-04 &1.99\\
&  16 & 1.81e-04 &  1.99  &6.03e-04 & 0.99 &1.56e-05 &1.99 & 5.87e-05&1.99\\
&  32 & 4.54e-05 & 1.99 &3.02e-04 & 0.99 &3.90e-06 & 1.99 & 1.47e-05& 1.99\\
\hline
\end{tabular}}
\caption{Convergence history for Example~\ref{ex1} with $K=1,\epsilon=1$.}
\label{ex1:table2}
\end{center}
\end{table}

\begin{table}[t]
\begin{center}
{\footnotesize
\begin{tabular}{cc||c c|c c|c c}
\hline
 &Mesh & \multicolumn{2}{|c|}{$\|\epsilon^{-1/2}(\bm{L}-\bm{L}_{h})\|_{0,\Omega_B}$} & \multicolumn{2}{|c|}{$\|\bm{u}_B-\bm{u}_{B,h}\|_{0,\Omega_B}$} & \multicolumn{2}{|c}{$\|p_B-p_{B,h}\|_{0,\Omega_B}$}\\
\hline
$k$ & $h^{-1}$  & Error & Order & Error& Order& Error & Order \\
\hline
1& 2  & 7.62e-07 &   N/A  &9.66e-04 &  N/A &5.48e-04 &   N/A \\
&  4  & 2.99e-07 &  1.34  &2.31e-04 & 2.06 &1.75e-04 & 1.65  \\
&  8  & 1.06e-07 &  1.50  &5.59e-05 & 2.05 &4.63e-05& 1.92 \\
&  16 & 3.62e-08 &  1.55   &1.38e-05 & 2.02 &1.17e-05 &1.98 \\
&  32 & 1.24e-08 & 1.55  &3.40e-06 & 2.00 &2.90e-06 & 1.99  \\
\hline
\end{tabular}}
\caption{Convergence history for Example~\ref{ex1} with $K=0.001,\epsilon=10^{-8}$.}
\label{ex1:table3}
\end{center}
\end{table}

\begin{table}[t]
\begin{center}
{\footnotesize
\begin{tabular}{cc||c c|c c|c c|c c}
\hline
 &Mesh & \multicolumn{2}{|c|}{$\|\bm{u}_D-\bm{u}_{D,h}\|_{0,\Omega_D}$} & \multicolumn{2}{|c|}{$\|p_D-p_{D,h}\|_{0,\Omega_D}$} & \multicolumn{2}{|c}{$\|c-c_h\|_{0}$} & \multicolumn{2}{|c}{$\|\bm{z}-\bm{z}_h\|_{0}$}\\
\hline
$k$ & $h^{-1}$  & Error & Order & Error& Order& Error & Order & Error & order \\
\hline
1& 2  & 9.00e-03 &   N/A  &4.60e-03 &  N/A &9.17e-04 &   N/A&7.74e-06 & NA \\
&  4  & 2.40e-03 &  1.93  &2.40e-03 & 0.95 &2.34e-04 & 1.97 &3.75e-06 & 1.04 \\
&  8  & 6.04e-04 &  1.97  &1.20e-03 & 0.99 &5.97e-05 & 1.97 &2.85e-06 &1.02\\
&  16 & 1.52e-04 &  1.99  &6.03e-04 & 0.99 &1.52e-05 &1.97 & 7.94e-07&1.21\\
&  32 & 3.83e-05 & 1.99 &3.02e-04 & 0.99 &3.90e-06 & 1.98 & 2.88e-07& 1.46\\
\hline
\end{tabular}}
\caption{Convergence history for Example~\ref{ex1} with $K=0.001,\epsilon=10^{-8}$.}
\label{ex1:table4}
\end{center}
\end{table}

\subsection{Example 2}\label{ex2}

In the second example we set $\Omega_B=(0,1/2)\times (0,1)$ and $\Omega_D=(1/2,1)\times (0,1)$ and
the velocity field is again continuous across the interface. The exact solution is defined by
\begin{align*}
 \bm{u}_D=\begin{cases}
\sin(2\pi x)\cos(2\pi y)\\
\sin(2\pi x)\cos(2\pi y)
\end{cases},\quad p_D=x(1/2-x)^2y(1-y)
\end{align*}
and
\begin{align*}
\bm{u}_B=\begin{cases}
x^2\sin(2\pi x)^2y^2\sin(\pi y)^2\\
x^2\sin(2\pi x)^2 y^2\sin(\pi y)^2
\end{cases}, \quad p_B= x(1/2-x)^2(y-1/2).
\end{align*}
The exact solution for transport equation is defined to be the same as Example~\ref{ex1}. We let $K_D=\alpha=1$. The convergence history for $L^2$ errors of all the variable against the meshsize is displayed in Table~\ref{ex2:table1}-Table~\ref{ex2:table4}. Various values of $K$ and $\epsilon$ are employed to test the robustness of the scheme. Similarly, we can observe that optimal convergence rates can be achieved for different values of $\epsilon$ and the scheme is uniformly robust for both the Stokes limit and Darcy limit. In addition, we can achieve second order convergence for diffusive flux when $K=1$ and the convergence rate will degenerate to first order when $K$ is small. This example once again highlights that the proposed scheme is uniformly robust for various values of viscosity.

\begin{table}[t]
\begin{center}
{\footnotesize
\begin{tabular}{cc||c c|c c|c c}
\hline
 &Mesh & \multicolumn{2}{|c|}{$\|\epsilon^{-1/2}(\bm{L}-\bm{L}_{h})\|_{0,\Omega_B}$} & \multicolumn{2}{|c|}{$\|\bm{u}_B-\bm{u}_{B,h}\|_{0,\Omega_B}$} & \multicolumn{2}{|c}{$\|p_B-p_{B,h}\|_{0,\Omega_B}$}\\
\hline
$k$ & $h^{-1}$  & Error & Order & Error& Order& Error & Order \\
\hline
1& 2  & 5.87e-02 &   N/A  &6.00e-03 &  N/A &2.36e-02 &   N/A \\
&  4  & 2.00e-02 &  1.55  &1.10e-03 & 2.39 &4.90e-03 & 2.26  \\
&  8  & 4.90e-03 &  2.01  &2.54e-04 & 2.17 &1.40e-03& 1.76 \\
&  16 &1.20e-03 &  1.99   &6.18e-05 & 2.04 &3.65e-04 &1.98 \\
&  32 & 3.11e-04 & 1.99  &1.53e-05 & 2.01 &9.08e-05 & 2.00  \\
\hline
\end{tabular}}
\caption{Convergence history for Example~\ref{ex2} with $K=1,\epsilon=1$.}
\label{ex2:table1}
\end{center}
\end{table}

\begin{table}[t]
\begin{center}
{\footnotesize
\begin{tabular}{cc||c c|c c|c c|c c}
\hline
 &Mesh & \multicolumn{2}{|c|}{$\|\bm{u}_D-\bm{u}_{D,h}\|_{0,\Omega_D}$} & \multicolumn{2}{|c|}{$\|p_D-p_{D,h}\|_{0,\Omega_D}$} & \multicolumn{2}{|c}{$\|c-c_h\|_{0}$} & \multicolumn{2}{|c}{$\|\bm{z}-\bm{z}_h\|_{0}$}\\
\hline
$k$ & $h^{-1}$  & Error & Order & Error& Order& Error & Order & Error & order \\
\hline
1& 2  & 1.40e-01 &   N/A  &1.34e-02 &  N/A &1.00e-03 &   N/A&5.00e-03 & NA \\
&  4  & 4.53e-02 &  1.63  &4.50e-03 & 1.56 &2.62e-04 & 1.96 &1.40e-03 & 1.84 \\
&  8  & 1.16e-02 &  1.96  &1.60e-03 & 1.50 &6.58e-05 & 1.99 &3.64e-04 &1.96\\
&  16 & 2.90e-03 &  1.99  &6.60e-04 & 1.27 &1.65e-05 &1.99 & 9.16e-05&1.99\\
&  32 & 7.33e-04 & 1.99 &3.09e-04 & 1.09 &4.12e-06 & 1.99 & 2.30e-05& 1.99\\
\hline
\end{tabular}}
\caption{Convergence history for Example~\ref{ex2} with $K=1,\epsilon=1$.}
\label{ex2:table2}
\end{center}
\end{table}

\begin{table}[t]
\begin{center}
{\footnotesize
\begin{tabular}{cc||c c|c c|c c}
\hline
 &Mesh & \multicolumn{2}{|c|}{$\|\epsilon^{-1/2}(\bm{L}-\bm{L}_{h})\|_{0,\Omega_B}$} & \multicolumn{2}{|c|}{$\|\bm{u}_B-\bm{u}_{B,h}\|_{0,\Omega_B}$} & \multicolumn{2}{|c}{$\|p_B-p_{B,h}\|_{0,\Omega_B}$}\\
\hline
$k$ & $h^{-1}$  & Error & Order & Error& Order& Error & Order \\
\hline
1& 2  & 2.79e-05 &   N/A  &3.59e-02 &  N/A &4.90e-03 &   N/A \\
&  4  & 1.05e-05 &  1.42  &7.40e-03 & 2.28 &9.23e-04 & 2.41  \\
&  8  & 3.73e-06 &  1.49  &1.80e-03 & 2.02 &2.37e-04& 1.96 \\
&  16 &1.28e-06 &  1.55   &4.45e-04 & 2.02 &5.97e-05 &1.99 \\
&  32 & 4.45e-07 & 1.52  &1.10e-04 & 2.02 &1.49e-05 & 2.00  \\
\hline
\end{tabular}}
\caption{Convergence history for Example~\ref{ex2} with $K=0.001,\epsilon=10^{-8}$.}
\label{ex2:table3}
\end{center}
\end{table}

\begin{table}[t]
\begin{center}
{\footnotesize
\begin{tabular}{cc||c c|c c|c c|c c}
\hline
 &Mesh & \multicolumn{2}{|c|}{$\|\bm{u}_D-\bm{u}_{D,h}\|_{0,\Omega_D}$} & \multicolumn{2}{|c|}{$\|p_D-p_{D,h}\|_{0,\Omega_D}$} & \multicolumn{2}{|c}{$\|c-c_h\|_{0}$} & \multicolumn{2}{|c}{$\|\bm{z}-\bm{z}_h\|_{0}$}\\
\hline
$k$ & $h^{-1}$  & Error & Order & Error& Order& Error & Order & Error & order \\
\hline
1& 2  & 1.24e-01 &   N/A  &1.08e-02 &  N/A &2.40e-03 &   N/A&4.13e-05 & NA \\
&  4  & 4.33e-02 &  1.52  &4.00e-03 & 1.43 &8.86e-04 & 1.43 &2.99e-05 & 0.47 \\
&  8  & 1.12e-02 &  1.95  &1.50e-03 & 1.43 &3.95e-04 & 1.59 &2.16e-05 &0.47\\
&  16 & 2.90e-03 &  1.99  &6.45e-04 & 1.21 &7.95e-05 &1.89 & 1.13e-05&0.94\\
&  32 & 7.07e-04 & 1.99 &3.07e-04 & 1.07 &1.74e-05 & 2.19 & 4.28e-06& 1.39\\
\hline
\end{tabular}}
\caption{Convergence history for Example~\ref{ex2} with $K=0.001,\epsilon=10^{-8}$.}
\label{ex2:table4}
\end{center}
\end{table}

\subsection{Example 3}\label{ex5}

In this example, the exact solution is unknown. This example illustrates the capability of the proposed method in the simulation of the groundwater flow. We consider $\Omega_B=(0,1)\times(1/2,1)$ and $\Omega_D=(0,1)\times(0,1/2)$. The Brinkman region represents a lake or a river, and the Darcy region represents an aquifer. We let the parameters set as $K = 10^{-5}$, $K_D=10^{-2}$, $\epsilon=0.1$, $\alpha=1$ and $c_{\text{in}}=0$. $\bm{f}_B=\bm{f}_D=\bm{0}$ and $f=0$. The time step size is chosen to be $\Delta t=10^{-3}$. Brinkman velocity is imposed everywhere on $\partial\Omega_B$ and it is set to be zero on the top and right boundary. On the left boundary we let $\bm{u}_B=(\frac{y(3/2-y)}{5},0)$. Zero normal velocity ($\bm{u}_D\cdot\bm{n}=0$) is imposed on the left and right boundary of $\Omega_D$ and Darcy pressure is imposed on the bottom, i.e., $p(x,0)=-0.05$.
In addition, the initial condition for $c$ is defined by
\begin{align*}
c^0=
\begin{cases}
1\quad \mbox{if}\; \sqrt{(x-0.1)^2+(y-0.7)^2}<0.1,\\
0\quad \mbox{otherwise}.
\end{cases}
\end{align*}
Although imposing Dirichlet boundary condition for Brinkman velocities is not covered in our analysis, we can modify the proposed scheme to adapt to this case. It should be noted that $\bm{L}\bm{n}=0$ is imposed on $\Gamma$ in order to ensure the unique solvability of the solution.

The profiles of the computed concentration at different time instants $t=1,3,6$ are displayed in Figure~\ref{ex51:ch}-Figure~\ref{ex53:ch}. We can observe that the concentration propagates from the surface water region to the groundwater region.
\begin{figure}[H]
\centering
\includegraphics[width=0.45\textwidth]{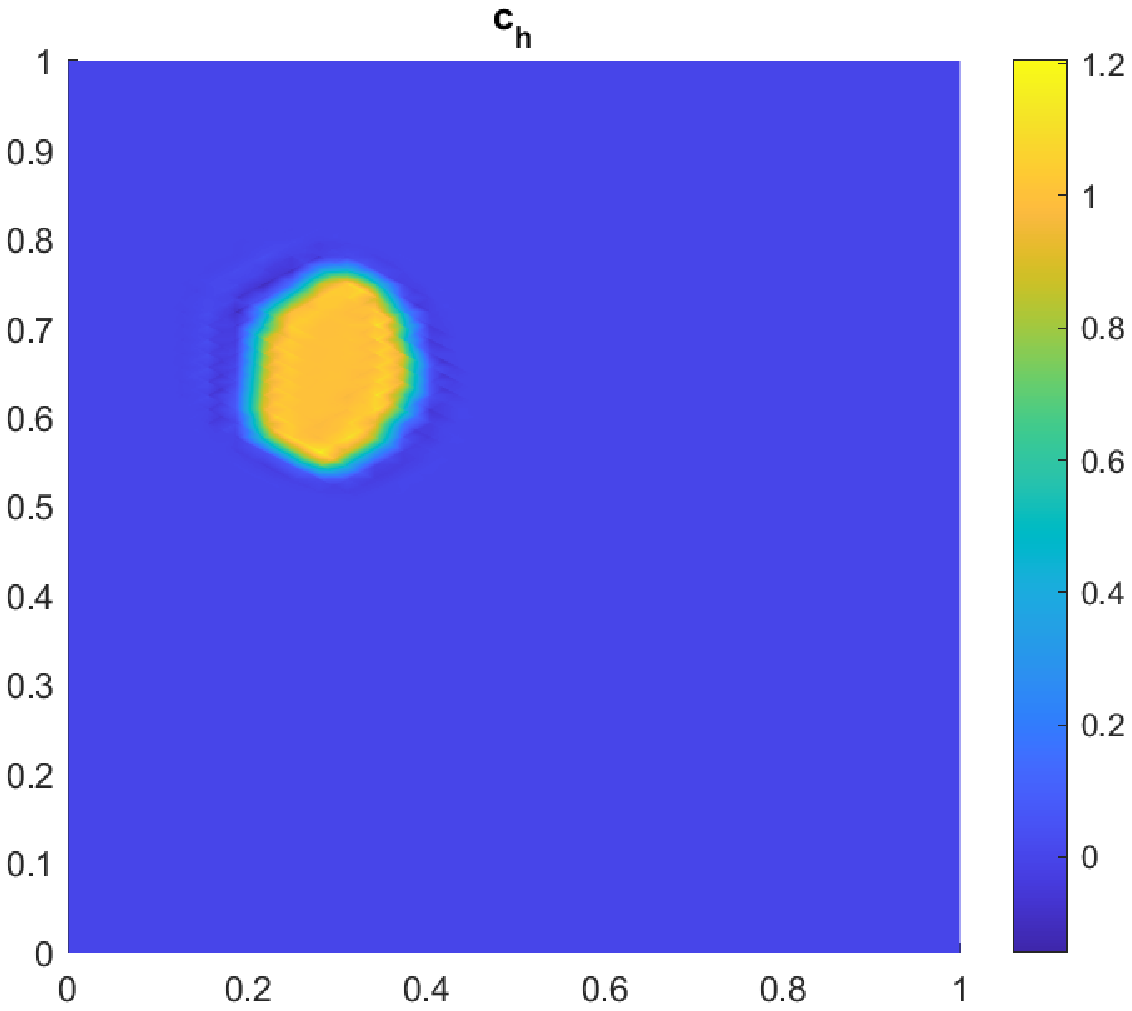}
\includegraphics[width=0.45\textwidth]{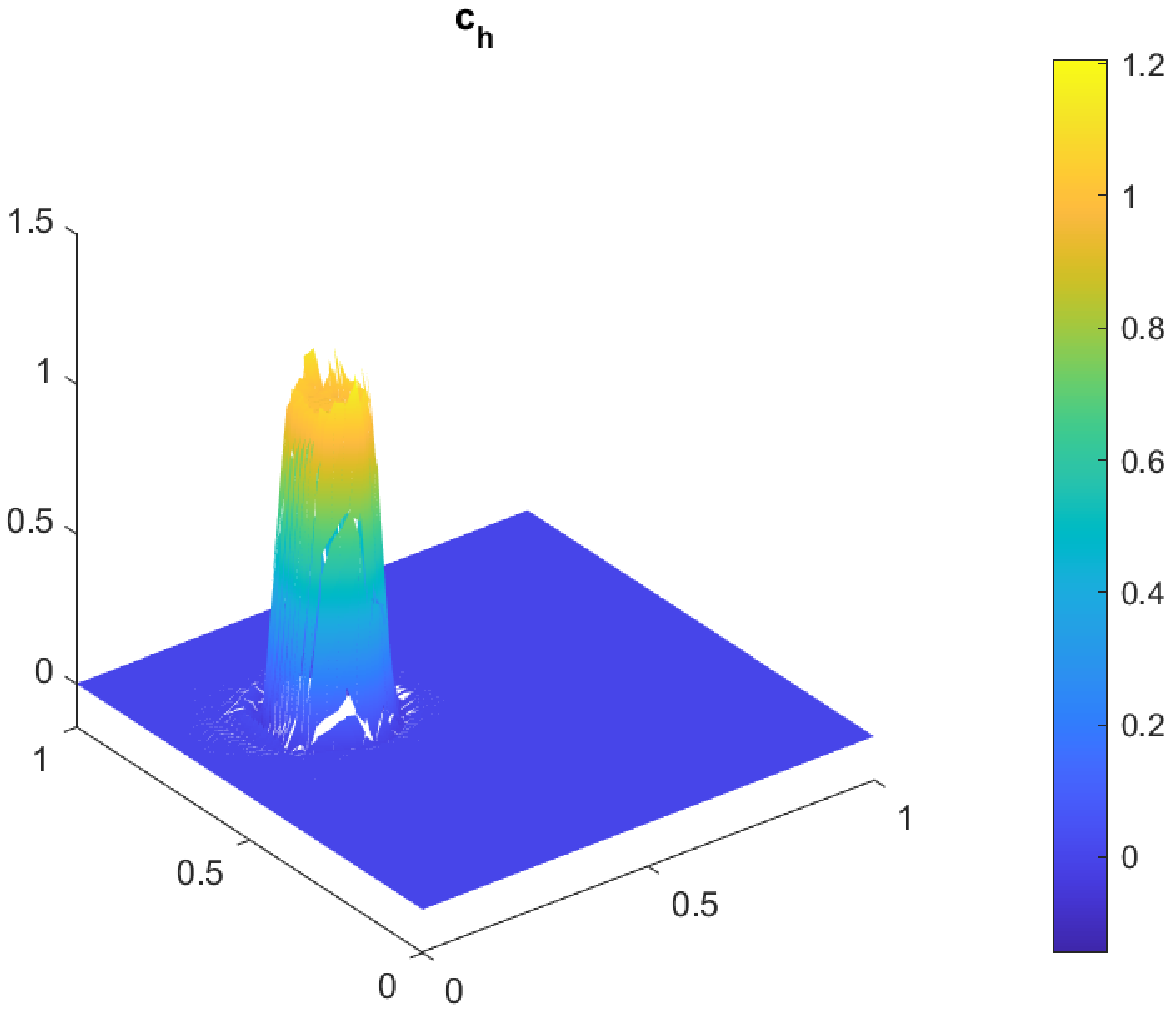}
\caption{Concentration at $t=1$ for Example~\ref{ex5}.}
\label{ex51:ch}
\end{figure}

\begin{figure}[H]
\centering
\includegraphics[width=0.45\textwidth]{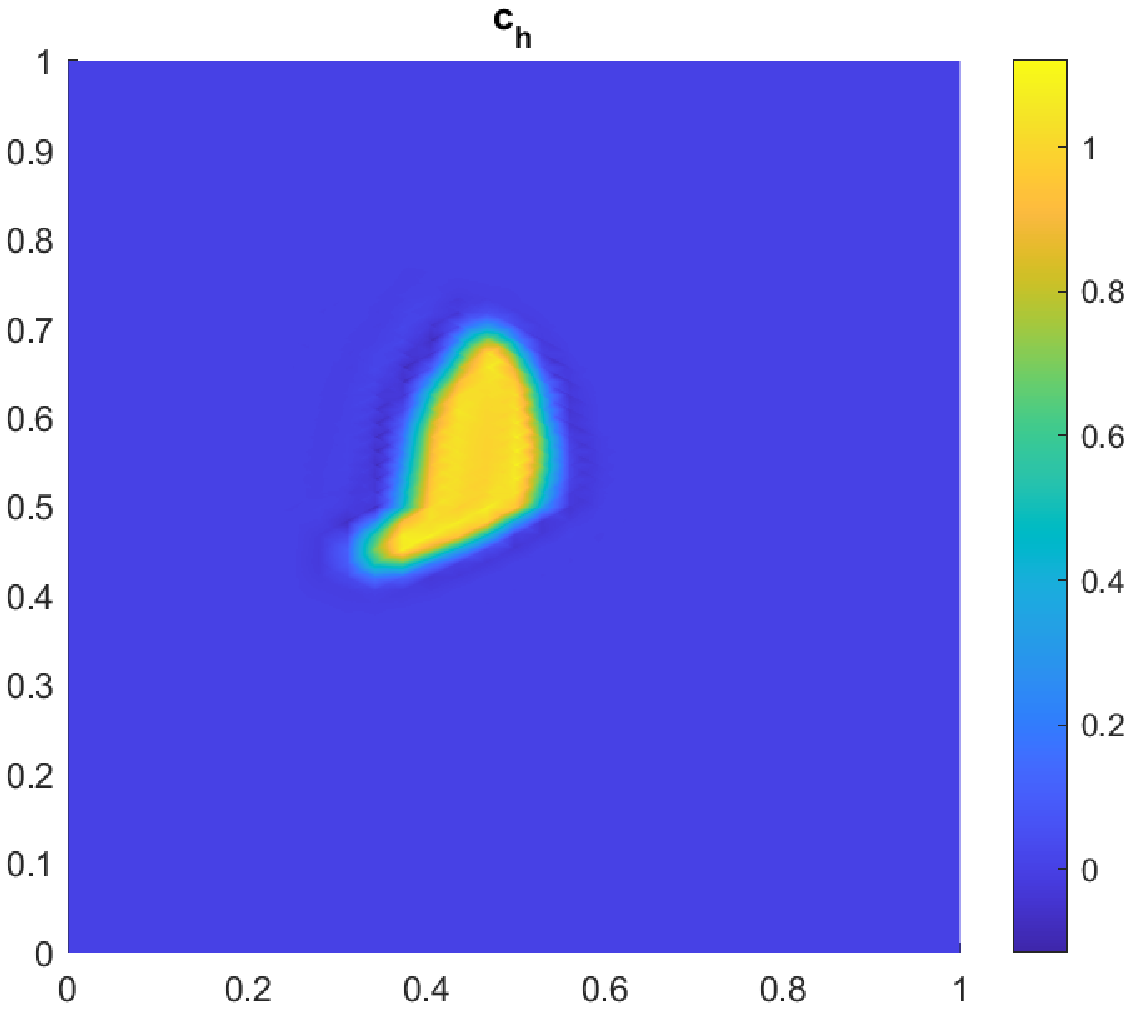}
\includegraphics[width=0.45\textwidth]{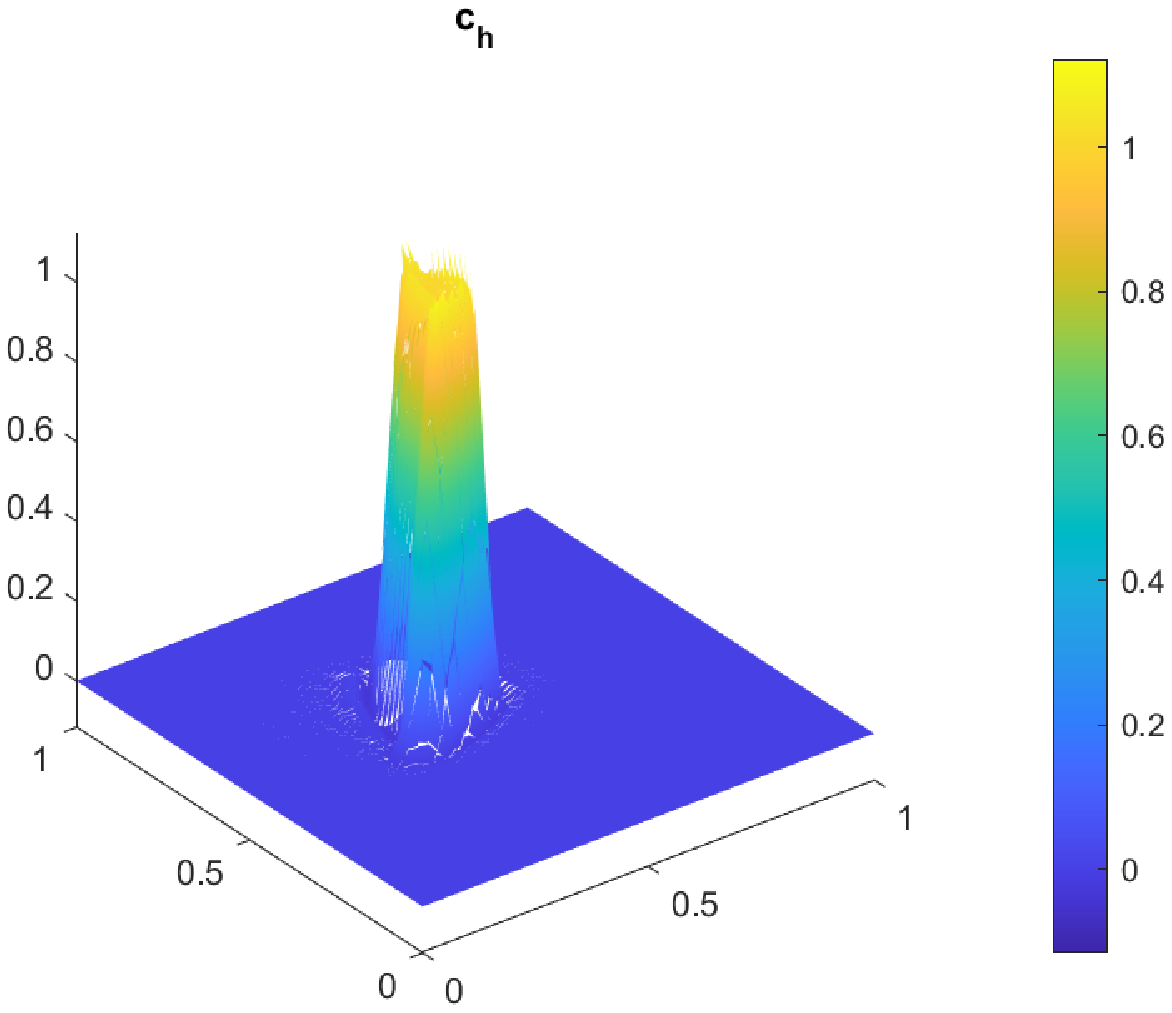}
\caption{Concentration at $t=3$ for Example~\ref{ex5}.}
\label{ex52:ch}
\end{figure}

\begin{figure}[H]
\centering
\includegraphics[width=0.45\textwidth]{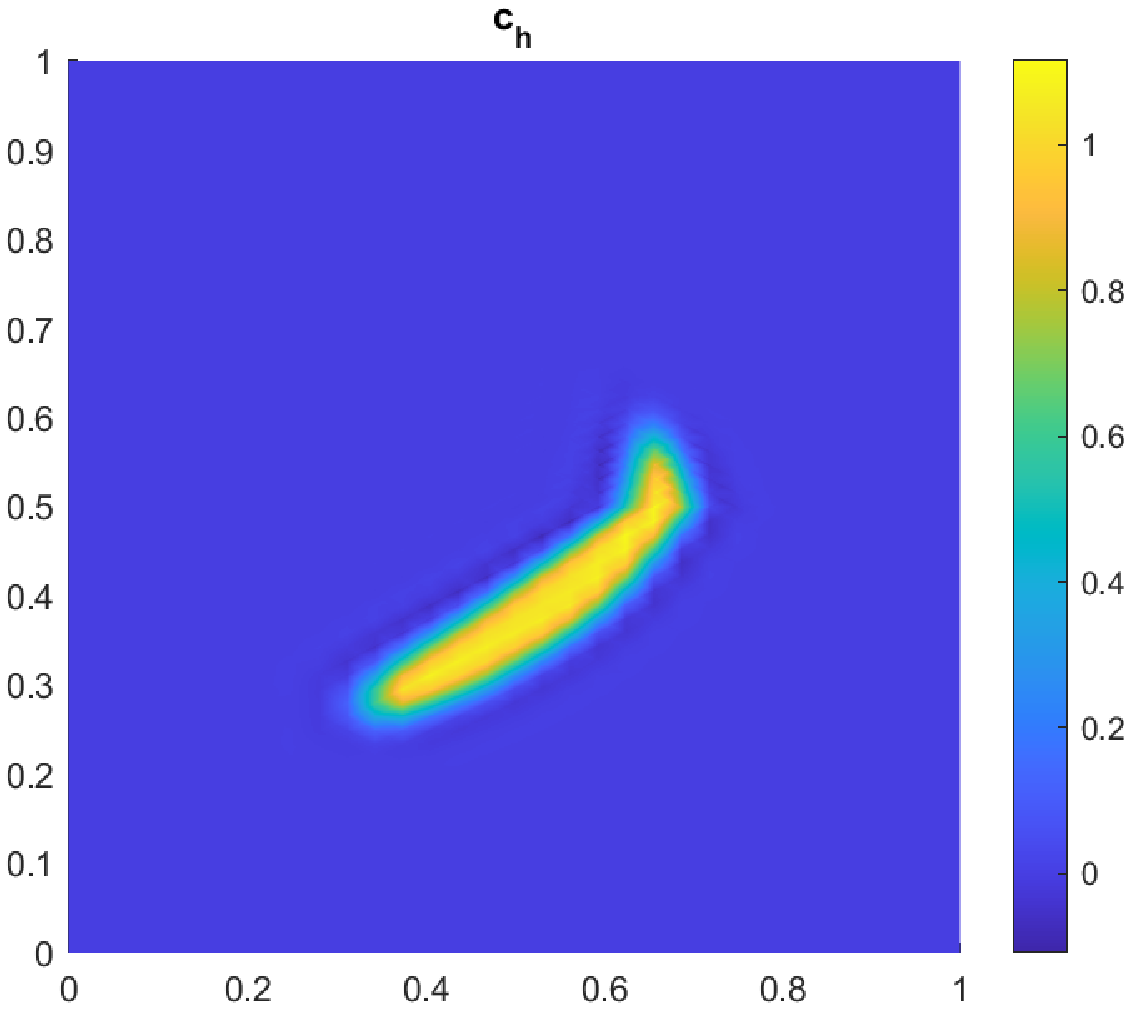}
\includegraphics[width=0.45\textwidth]{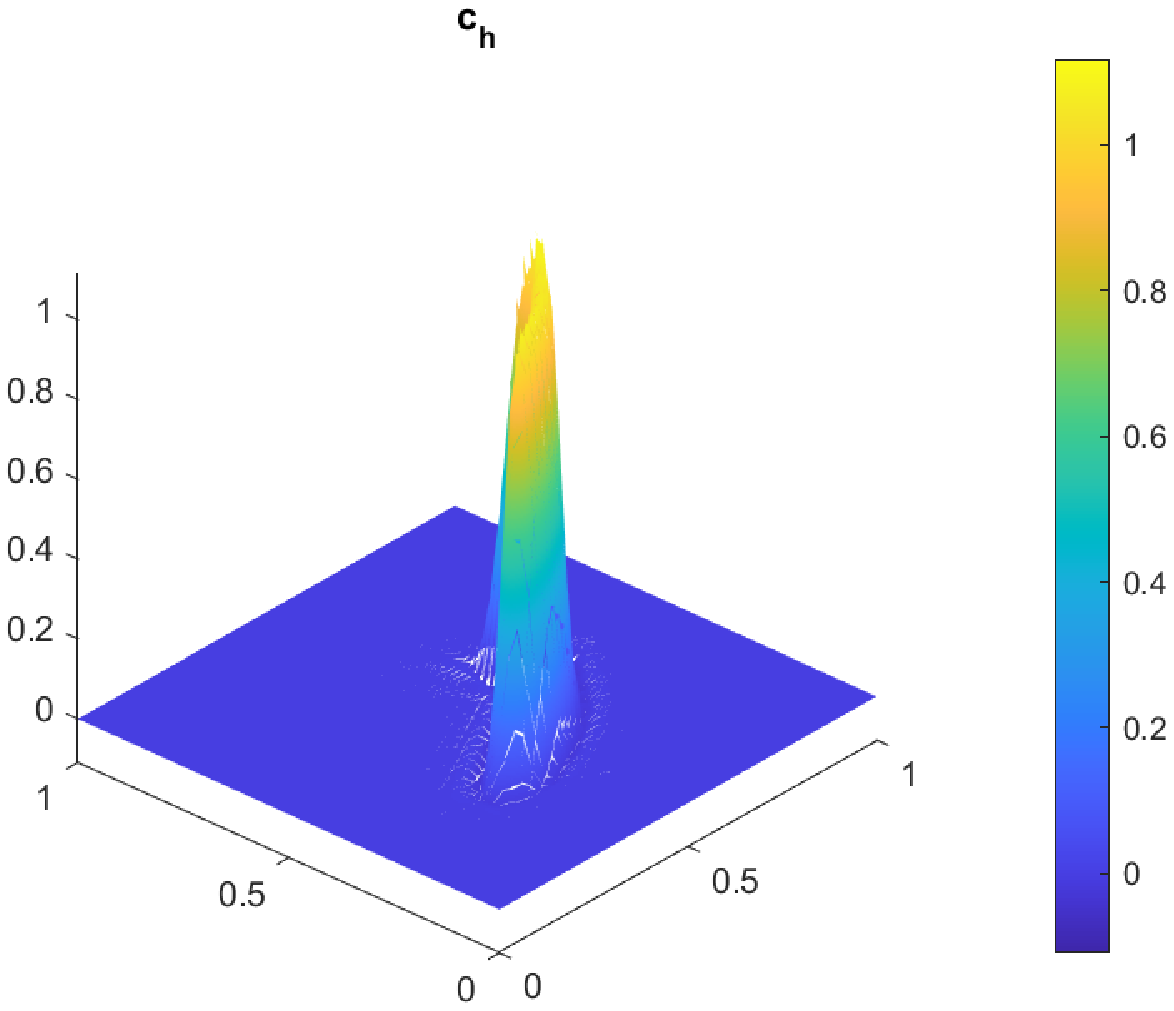}
\caption{Concentration at $t=6$ for Example~\ref{ex5}.}
\label{ex53:ch}
\end{figure}

\subsection{Example 4}\label{ex4}

In this example, the exact solution is also unknown. The computational domain corresponds to the rectangle $\Omega=(0,12)\times (0,6)$, where the Brinkman domain (with a maximum height of 4) is on the top and the Darcy subdomain (with a maximum height of 2.25) on the bottom. The two subdomains are separated by a step-polygonal interface and we use triangular meshes as the primal partition; see Figure~\ref{ex4:mesh} for an illustration. Note that the triangular meshes are generated using distmesh2d, cf. \cite{Persson04} and we use finer meshes near the interface. Each initial triangular mesh is subdivided into the union of triangles for the construction of the method. We let $\epsilon=1$, the permeability $K_D$ is selected as random numbers between $10^{-3}$ and $10^{-6}$, and $K=\alpha=1$. On the top segment of $\Gamma_B$, normal velocities are set to be zero, whereas on the left and right hand sides of the Brinkman domain we prescribe the following conditions
\begin{align*}
\bm{u}_B\cdot\bm{n}=\frac{1}{4}(y-4)(8-y)\quad \mbox{and}\quad \bm{u}_B\cdot\bm{n}=\frac{3}{16}(y-4)(8-y),
\end{align*}
respectively. Zero normal velocities are imposed for the vertical boundaries of $\Omega_D$ and Dirichlet pressure is imposed on the bottom, i.e., $p(x,0)=-10^3$. In addition, we let $\bm{f}_B=\bm{f}_D=\bm{0}$ and $f=0$. The parameters are set as $\phi=1, s=0.01$ and $c_{\text{in}}=1$ for the transport equation. The time step size is chosen to be $\Delta t=10^{-3}$. The approximated velocity and pressure are shown in Figure~\ref{ex4:uh-ph}. The approximated concentration at time $T=1,5,10,20$ are shown in Figure~\ref{ex4:ch} and we can observe that the concentration propagates to the right.

\begin{figure}[H]
\centering
\setlength{\abovecaptionskip}{-0.7cm}
\includegraphics[width=0.45\textwidth]{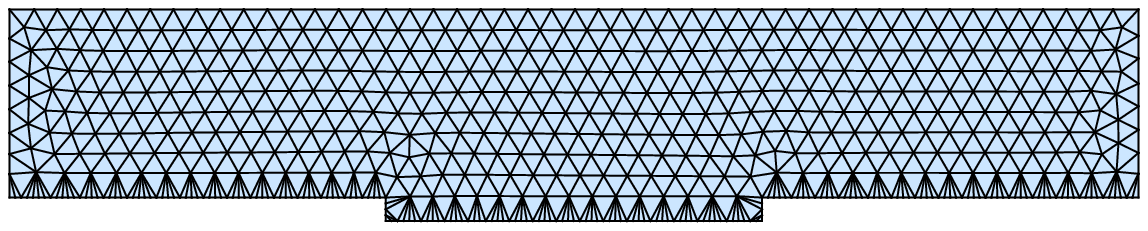}
\includegraphics[width=0.45\textwidth]{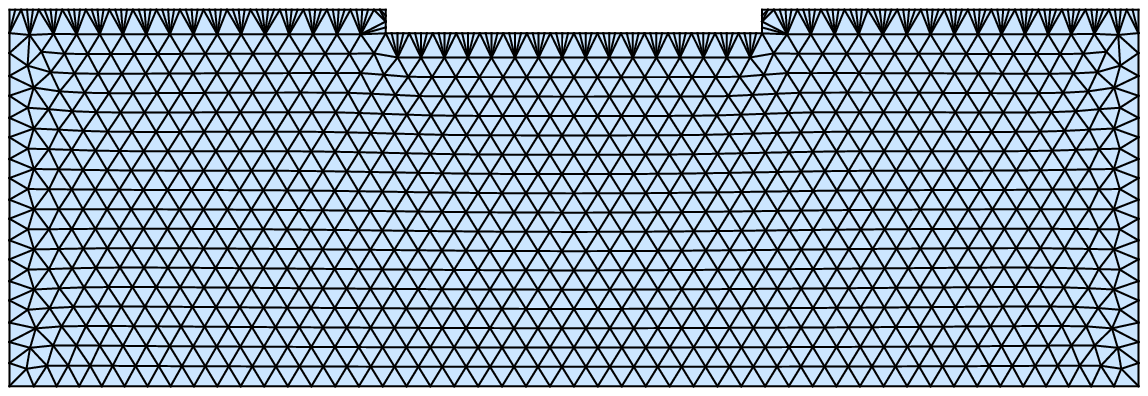}
\caption{Profile of the meshes used for Example~\ref{ex4}. The meshes for $\Omega_B$ (left) and the meshes for $\Omega_D$.}
\label{ex4:mesh}
\end{figure}

\begin{figure}[H]
\centering
\setlength{\abovecaptionskip}{-0.7cm}
\includegraphics[width=0.45\textwidth]{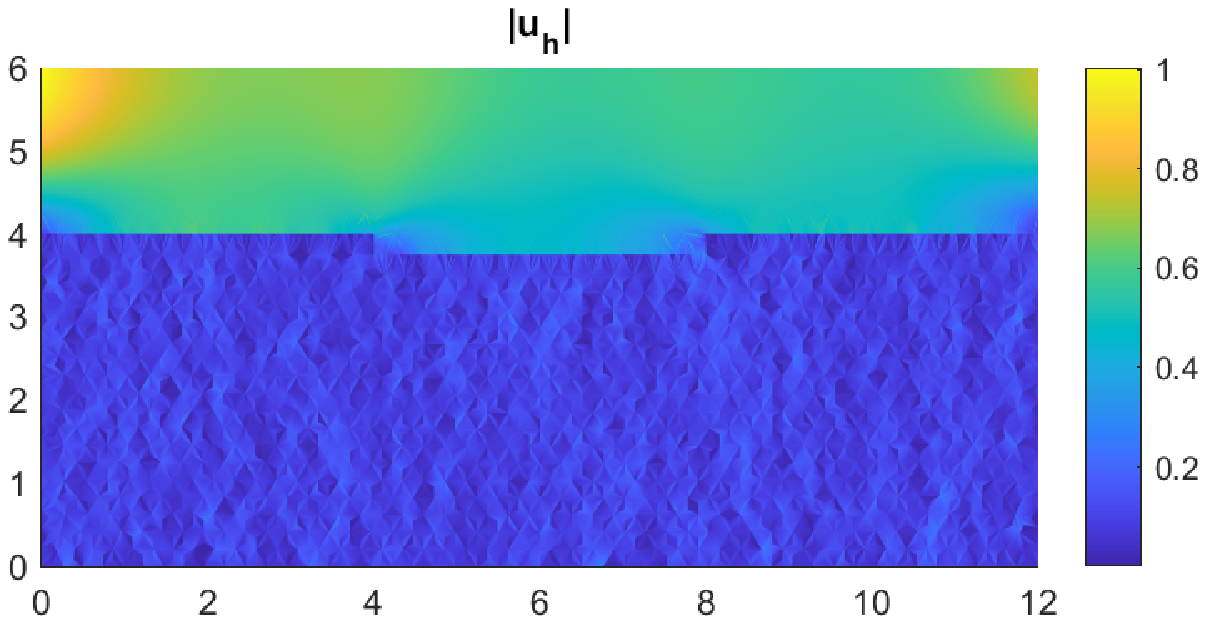}
\includegraphics[width=0.45\textwidth]{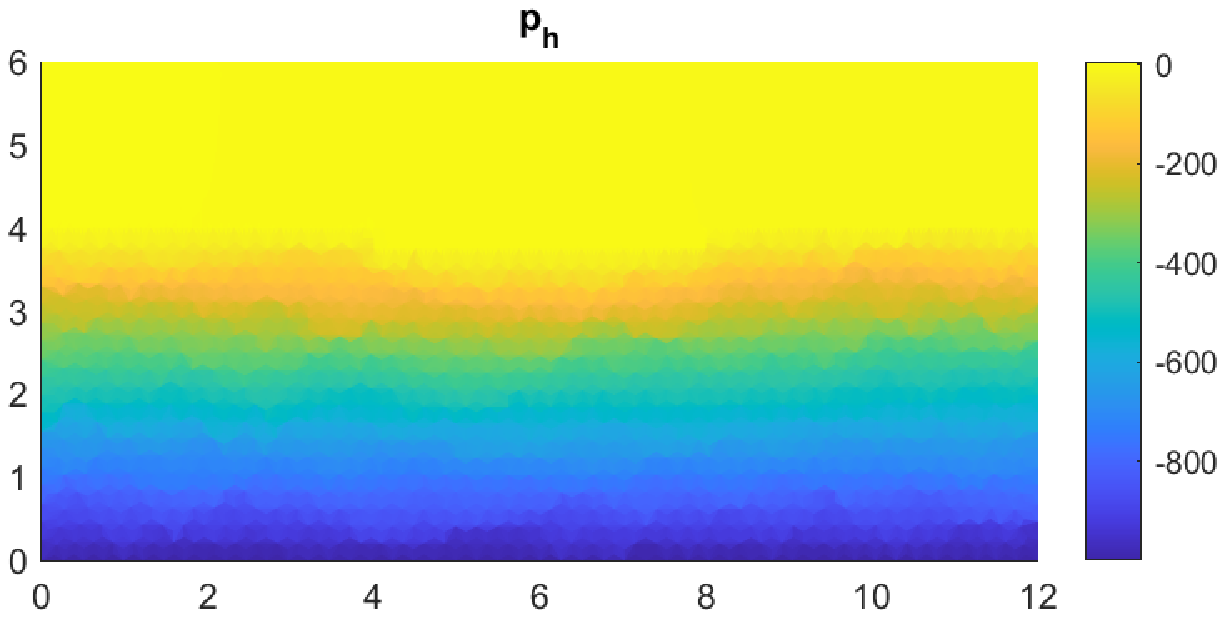}
\caption{Velocity magnitude $\bm{u}_h$ (left) and pressure (right) for Example~\ref{ex4}.}
\label{ex4:uh-ph}
\end{figure}

\begin{figure}[H]
\centering
\setlength{\abovecaptionskip}{-0.7cm}
\includegraphics[width=0.45\textwidth]{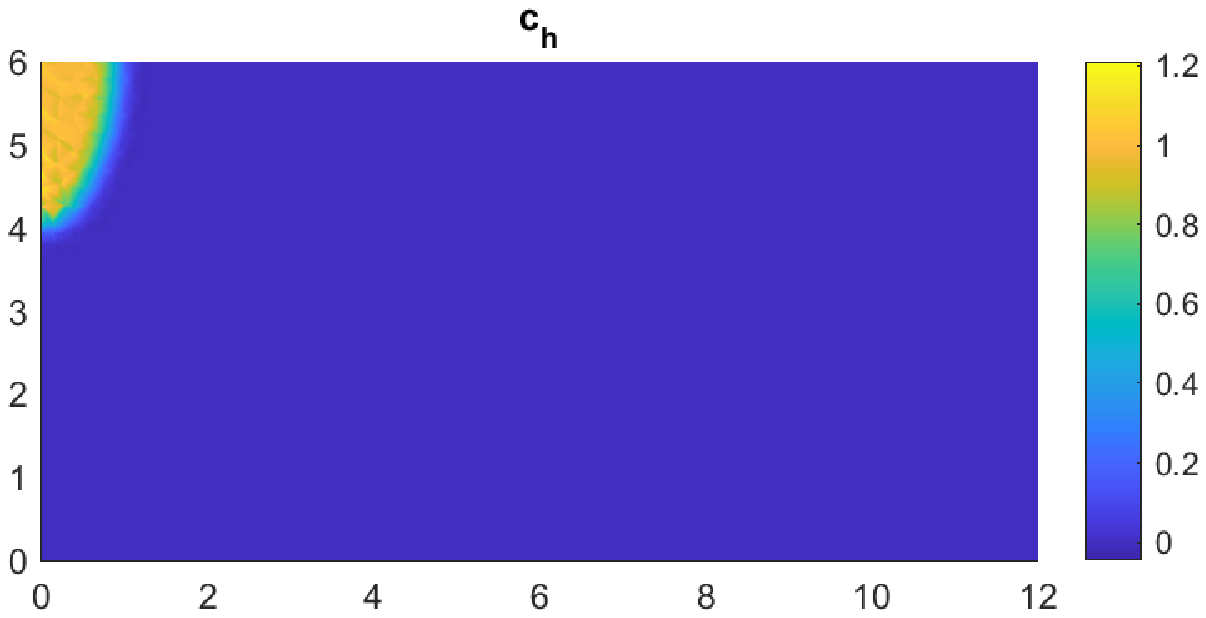}
\includegraphics[width=0.45\textwidth]{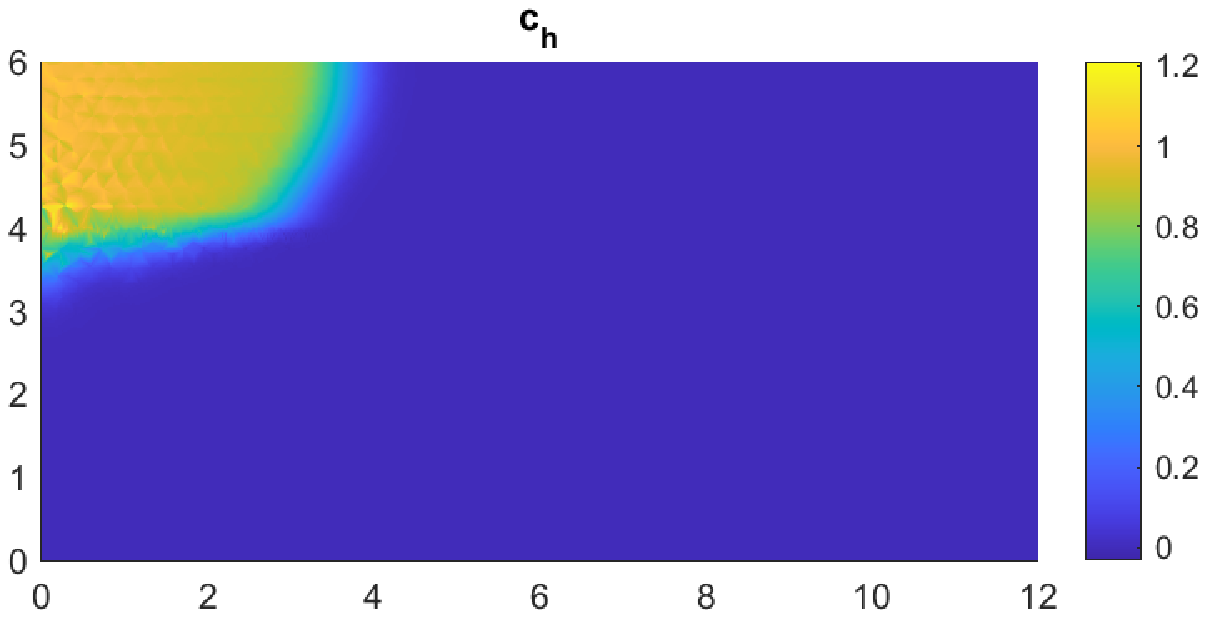}
\caption{Concentration at $t=1$ and $t=5$ for Example~\ref{ex4}.}
\label{ex4:ch}
\end{figure}

\begin{figure}[H]
\centering
\setlength{\abovecaptionskip}{-0.7cm}
\includegraphics[width=0.45\textwidth]{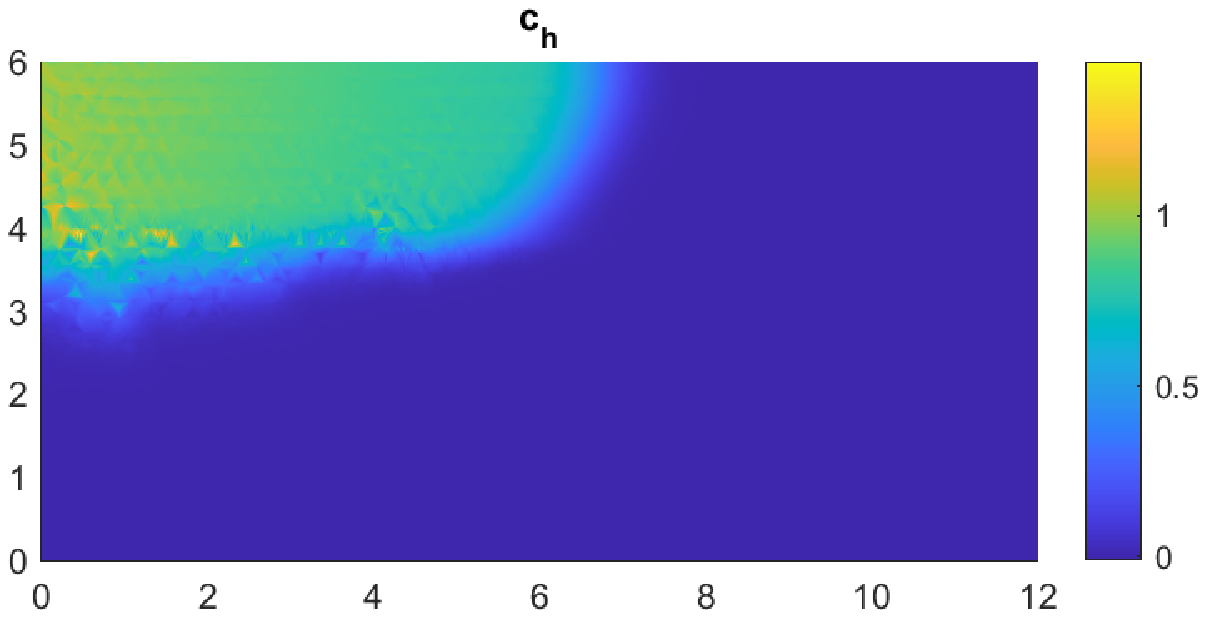}
\includegraphics[width=0.45\textwidth]{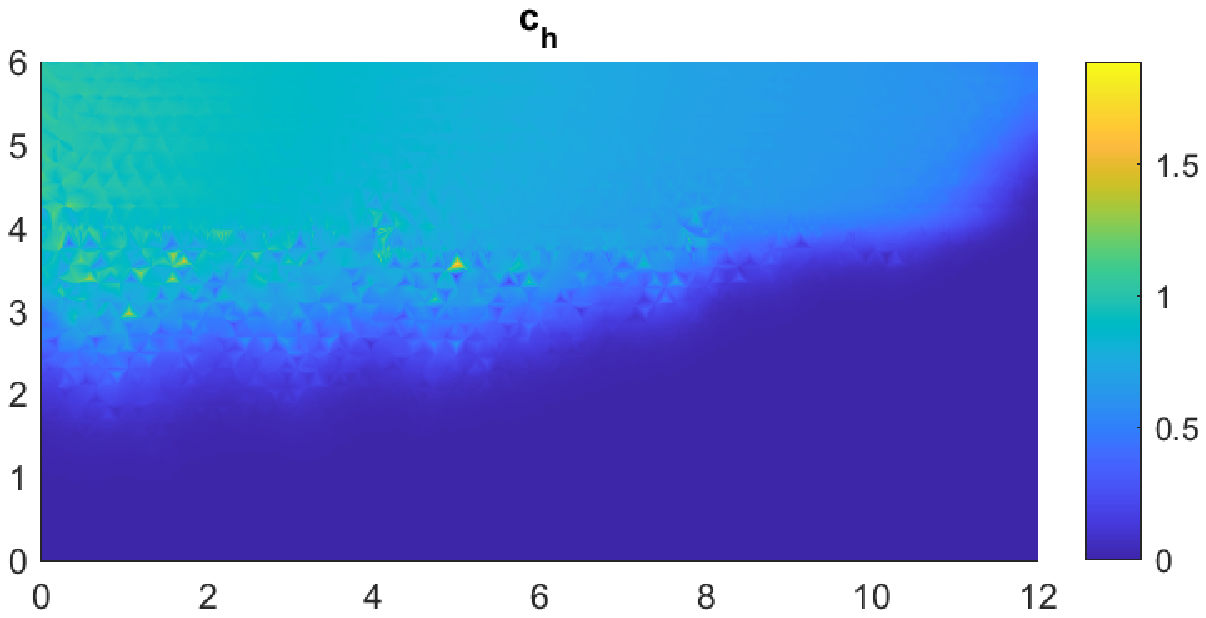}
\caption{Concentration at $t=10$ and $t=20$ for Example~\ref{ex4}.}
\label{ex42:ch}
\end{figure}

\section{Conclusion}\label{sec:conclusion}

In this paper we have designed a strongly  mass conservative scheme for the Brinkman-Darcy flow, where the interface conditions are enforced naturally in the discrete formulation. Theoretical analysis indicates that the proposed scheme is exactly divergence free in the Brinkman region and it is robust for both the Stokes limit and Darcy limit. Taking advantage of the mass conservation property, we design an upwinding staggered DG method for the transport equation, where the boundary correction terms are introduced to improve the stability estimate. Several numerical experiments illustrate that our scheme is indeed robust for both the Stokes limit and Darcy limit; in addition, optimal convergence rates can be achieved for various values of $\epsilon$. It is worth mentioning that the proposed scheme is pressure-robust and strongly mass conservative, which makes it a good candidate for the numerical simulation of coupled flow and transport.

\section*{Acknowledgments}

The research of Lina Zhao was supported by a grant from City University of Hong Kong (Project No. 7200699). The research of Shuyu Sun was supported by King Abdullah University of Science and Technology (KAUST, Saudi Arabia through the grants BAS/1/1351-01,URF/1/4074-01, and URF/1/3769-01).

\bibliographystyle{plain}
\bibliography{reference}

\begin{thebibliography}{10}

\bibitem{Alvarez16}
M.~Alvarez, G.~N. Gatica, and R.~Ruiz-Baier.
\newblock A vorticity-based fully-mixed formulation for the 3{D}
  {B}rinkman-{D}arcy problem.
\newblock {\em Comput. Methods Appl. Mech. Engrg.}, 307:68--95, 2016.

\bibitem{Alvarez20}
M.~Alvarez, G.~N. Gatica, and R.~Ruiz-Baier.
\newblock A mixed-primal finite element method for the coupling of
  {Brinkman}-{Darcy} flow and nonlinear transport.
\newblock {\em IMA J. Numer. Anal.}, 2020.

\bibitem{Apel21}
T.~Apel, V.~Kempf, A.~Linke, and C.~Merdon.
\newblock {A nonconforming pressure-robust finite element method for the
  {S}tokes equations on anisotropic meshes}.
\newblock {\em IMA J. Numer. Anal.}, 2021.

\bibitem{Brezzi85}
F.~Brezzi, J.~Douglas, and L.~D. Marini.
\newblock Two families of mixed finite elements for second order elliptic
  problems.
\newblock {\em Numer. Math.}, 47:217--235, 1985.

\bibitem{Chidyagwai09}
P.~Chidyagwai and B.~Rivi\`{e}re.
\newblock On the solution of the coupled {N}avier-{S}tokes and {D}arcy
  equations.
\newblock {\em Comput. Methods Appl. Mech. Engrg.}, 198:3806--3820, 2009.

\bibitem{ChungEngquist09}
E.~T. Chung and B.~Engquist.
\newblock Optimal discontinuous {Galerkin} methods for the acoustic wave
  equation in higher dimensions.
\newblock {\em SIAM J. Numer. Anal.}, 47(5):3820--3848, 2009.

\bibitem{Ciarlet78}
P.~G. Ciarlet.
\newblock {\em The Finite Element Method for Elliptic Problems}.
\newblock North-Holland Publishing, Amsterdam, 1978.

\bibitem{Cockburn99}
B.~Cockburn and C.~Dawson.
\newblock {\em Some extensions of the local discontinuous Galerkin method for
  convection-diffusion equations in multidimensions}.
\newblock X, MAFELAP 1999 (Uxbridge), Elsevier, Oxford, 2000, pp. 225--238.

\bibitem{Dawson04}
C.~Dawson, S.~Sun, and M.~F. Wheeler.
\newblock Compatible algorithms for coupled flow and transport.
\newblock {\em Comput. Methods Appl. Mech. Engrg.}, 193:2565--2580, 2004.

\bibitem{Duran88}
R.~G. Dur\'an.
\newblock Error analysis in ${L}^p \leqslant p \leqslant \infty $, for mixed
  finite element methods for linear and quasi-linear elliptic problems.
\newblock {\em ESAIM: Math. Model. Numer. Anal}, 22:371--387, 1988.

\bibitem{Ervin19}
V.~Ervin, M.~Kubacki, W.~Layton, M.~Moraiti, Z.~Si, and C.~Trenchea.
\newblock Partitioned penalty methods for the transport equation in the
  evolutionary {S}tokes-{D}arcy-{t}ransport problem.
\newblock {\em Numer. Meth. Part Differ. Equ.}, 35:349--374, 2019.

\bibitem{Ervin09}
V.~J. Ervin, E.~W. Jenkins, and S.~Sun.
\newblock Coupled generalized nonlinear {Stokes} flow with flow through a
  porous medium.
\newblock {\em SIAM J. Numer. Anal.}, 47:929--952, 2009.

\bibitem{Feng02}
X.~Feng and O.~A. Karakashian.
\newblock Two-level additive schwarz methods for a discontinuous galerkin
  approximation of second order elliptic problems.
\newblock {\em SIAM J. Numer. Anal.}, 39(4):1343--1365, 2002.

\bibitem{Frerichs20}
D.~Frerichs and C.~Merdon.
\newblock {Divergence-preserving reconstructions on polygons and a really
  pressure-robust virtual element method for the {S}tokes problem}.
\newblock {\em IMA J. Numer. Anal.}, 2020.

\bibitem{FuQiu18}
G.~Fu, Y.~Jin, and W.~Qiu.
\newblock Parameter-free superconvergent ${H}(\text{div})$-conforming {HDG}
  methods for the {Brinkman} equations.
\newblock {\em IMA J. Numer. Anal.}, 39:957--982, 2018.

\bibitem{Fu18}
G.~Fu and C.~Lehrenfeld.
\newblock A strongly conservative hybrid {DG}/mixed {FEM} for the coupling of
  {S}tokes and {D}arcy flow.
\newblock {\em J. Sci. Comput.}, 77:1605--1620, 2018.

\bibitem{Gatica14}
G.~N. Gatica.
\newblock {\em A {Simple} {Introduction} to the {Mixed} {Finite} {Element}
  {Method}: Theory and Applications}.
\newblock Springer Briefs in Mathematics, Springer, Cham, 2014.

\bibitem{Gatica11}
G.~N. Gatica, R.~Oyarz\'{u}a, and F.-J. Sayas.
\newblock Analysis of fully-mixed finite element methods for the
  {S}tokes-{D}arcy coupled problem.
\newblock {\em Math. Comp.}, 80:1911--1948, 2011.

\bibitem{GiraultRaviart86}
V.~Girault and P.-A. Raviart.
\newblock {\em Finite Element Methods for Navier-Stokes Equations}.
\newblock Springer-Verlag Berlin, 1986.

\bibitem{Guzman12}
J.~Guzm\'{a}n and M.~Neilan.
\newblock {A family of nonconforming elements for the {B}rinkman problem}.
\newblock {\em IMA J. Numer. Anal.}, 32:1484--1508, 2012.

\bibitem{Konno11}
J.~K\"{o}nn\"{o} and R.~Stenberg.
\newblock H(div)-conforming finite elements for the {B}rinkman problem.
\newblock {\em Math. Models Meth. Appl. Sci.}, 21:2227--2248, 2011.

\bibitem{Konno12}
J.~K\"{o}nn\"{o} and R.~Stenberg.
\newblock Numerical computations with {H}(div)-finite elements for the
  {B}rinkman problem.
\newblock {\em Computat. Geosci.}, 16:139--158, 2012.

\bibitem{Layton02}
W.~J. Layton, F.~Schieweck, and I.~Yotov.
\newblock Coupling fluid flow with porous media flow.
\newblock {\em SIAM J. Numer. Anal.}, 40:2195--2218, 2002.

\bibitem{Lederer17}
P.~L. Lederer, A.~Linke, C.~Merdon, and J.~Sch\"{o}berl.
\newblock Divergence-free reconstruction operators for pressure-robust {S}tokes
  discretizations with continuous pressure finite elements.
\newblock {\em SIAM J. Numer. Anal.}, 55:1291--1314, 2017.

\bibitem{Lee16}
S.~Lee, Y.-J. Lee, and M.~F. Wheeler.
\newblock A locally conservative enriched {G}alerkin approximation and
  efficient solver for elliptic and parabolic problems.
\newblock {\em SIAM J. Sci. Comput.}, 38(3):A1404--A1429, 2016.

\bibitem{Linke12}
A.~Linke.
\newblock A divergence-free velocity reconstruction for incompressible flows.
\newblock {\em C. R. Math. Acad. Sci. Paris}, 350:837--840, 2012.

\bibitem{Linke14}
A.~Linke.
\newblock On the role of the {H}elmholtz decomposition in mixed methods for
  incompressible flows and a new variational crime.
\newblock {\em Comput. Methods Appl. Mech. Engrg.}, 268:782--800, 2014.

\bibitem{Lipnikov14}
K.~Lipnikov, D.~Vassilev, and I.~Yotov.
\newblock Discontinuous {G}alerkin and mimetic finite difference methods for
  coupled {S}tokes-{D}arcy flows on polygonal and polyhedral grids.
\newblock {\em Numer. Math.}, 193:321--360, 2014.

\bibitem{Mardal02}
K.~A. Mardal, X.-C. Tai, and R.~Winther.
\newblock A robust finite element method for {D}arcy--{S}tokes flow.
\newblock {\em SIAM J. Numer. Anal.}, 40:1605--1631, 2002.

\bibitem{Mu20}
L.~Mu.
\newblock A uniformly robust ${H(\text{div})}$ weak {Galerkin} finite element
  methods for {Brinkman} problems.
\newblock {\em SIAM J. Numer. Anal.}, 58:1422--1439, 2020.

\bibitem{Persson04}
P.-O. Persson and G.~Strang.
\newblock A simple mesh generator in {MATLAB}.
\newblock {\em SIAM Review}, 46:329--345, 2004.

\bibitem{Raviart77}
P.~A. Raviart and J.~M. Thomas.
\newblock A mixed finite element method for 2-nd order elliptic problems.
\newblock In {\em Mathematical Aspects of Finite Element Methods}. Springer
  Berlin Heidelberg, 1977.

\bibitem{Riviere05}
B.~Rivi\`{e}re and I.~Yotov.
\newblock Locally conservative coupling of {S}tokes and {D}arcy flows.
\newblock {\em SIAM J. Numer. Anal.}, 42:1959--1977, 2005.

\bibitem{Rui17}
H.~Rui and J.~Zhang.
\newblock A stabilized mixed finite element method for coupled {Stokes} and
  {Darcy} flows with transport.
\newblock {\em Comput. Methods Appl. Mech. Engrg.}, 315:169--189, 2017.

\bibitem{SunWheeler05}
S.~Sun and M.~F. Wheeler.
\newblock Discontinuous {Galerkin} methods for coupled flow and reactive
  transport problems.
\newblock {\em Appl. Numer. Math.}, 52:273--298, 2005.

\bibitem{Vassilev09}
D.~Vassilev and I.~Yotov.
\newblock Coupling {Stokes}-{Darcy} flow with transport.
\newblock {\em SIAM J. Sci. Comput.}, 31:3661--3684, 2009.

\bibitem{Wang21}
G.~Wang, L.~Mu, Y.~Wang, and Y.~He.
\newblock A pressure-robust virtual element method for the {S}tokes problem.
\newblock {\em Comput. Methods Appl. Mech. Engrg.}, 382:113879, 2021.

\bibitem{Zhang18}
J.~Zhang, H.~Rui, and Y.~Cao.
\newblock A partitioned method with different time steps for coupled {S}tokes
  and {D}arcy flows with transport.
\newblock {\em Int. J. Numer. Anal. Model.}, 16:463--498, 2018.

\bibitem{ZhaoChungLam20}
L.~Zhao, E.~Chung, and M.~Lam.
\newblock A new staggered {DG} method for the {Brinkman} problem robust in the
  {Darcy} and {Stokes} limits.
\newblock {\em Comput. Methods Appl. Mech. Engrg.}, 364, 2020.

\bibitem{Zhao2018}
L.~Zhao and E.-J. Park.
\newblock A staggered discontinuous {G}alerkin method of minimal dimension on
  quadrilateral and polygonal meshes.
\newblock {\em SIAM J. Sci. Comput.}, 40:A2543--A2567, 2018.

\bibitem{Zhao20}
L.~Zhao and E.-J. Park.
\newblock {A lowest-order staggered DG method for the coupled {S}tokes-{D}arcy
  problem}.
\newblock {\em IMA J. Numer. Anal.}, 40(4):2871--2897, 2020.

\bibitem{LinaParkhybrid20}
L.~Zhao and E.-J. Park.
\newblock A new hybrid staggered discontinuous {Galerkin} method on general
  meshes.
\newblock {\em J. Sci. Comput.}, 82, 2020.

\bibitem{ZHAO2019}
L.~Zhao, E.-J. Park, and D.~w.~Shin.
\newblock A staggered {DG} method of minimal dimension for the {S}tokes
  equations on general meshes.
\newblock {\em Comput. Methods Appl. Mech. Engrg.}, 345:854--875, 2019.

\end{thebibliography}

\end{document}